\documentclass[12pt]{article}

\usepackage{amssymb}
\usepackage{amsmath}
\usepackage{amsthm}

\usepackage{mathrsfs}

\usepackage[hidelinks]{hyperref}

\newtheorem{thm}{Theorem}[section]
\newtheorem{cor}[thm]{Corollary}
\newtheorem{lem}[thm]{Lemma}
\newtheorem{prp}[thm]{Proposition}

\theoremstyle{definition}

\newtheorem{exa}[thm]{Example}

\newtheorem{rem}[thm]{Remark}

\newcommand{\scr}[1]{\mathscr #1}

\numberwithin{equation}{section}

\def\eins{\boldsymbol 1}
\def\R{\mathbb R}  \def\ff{\frac} \def\ss{\sqrt} \def\B{\mathbf B}
\def\N{\mathbb N}  

 \def\DD{\Delta} \def\vv{\varepsilon} \def\rr{\rho}
\def\<{\langle} \def\>{\rangle}
\def\nn{\nabla} \def\pp{\partial} \def\E{\mathbb E}
\def\d{\text{\rm{d}}}   \def\D{\scr D}
\def\si{\sigma} 

\def\e{\text{\rm{e}}}    
\def\tt{\widetilde}\def\[{\lfloor} \def\]{\rfloor}
 \def\P{\mathbb P} 
\def\C{\scr C}           
   
  \def\ll{\lambda}
 
  \def\LL{\Lambda}
\def\to{\rightarrow}\def\gg{\gamma}
\def\EE{\scr E} \def\W{\mathbb W}
 
\def\B{\scr B}

   \def\supp{{\rm supp}}
\def\1{{\bf 1}}
 \def\beg{\begin} 
\def\beq{\beg{equation}}

\date{}

\begin{document}

\title{{\bf  Diffusion Processes  on p-Wasserstein Space over Banach Space}\footnote{Feng-Yu Wang is supported in part by the National Key R\& D Program of China (No. 2022YFA1006000, 2020YFA0712900), NNSFC (12531007) and
	Deutsche Forschungsgemeinschaft (DFG, German
	Research Foundation) – Project-ID 317210226 – SFB 1283.
	Panpan Ren is supported by NNSFC (12301180),  RGC(21301925) and Research Centre for Nonlinear Analysis at Hong Kong PolyU.  Simon Wittmann is supported by Research Centre for Nonlinear Analysis at Hong Kong PolyU.}}
\author{Panpan Ren$^{(b)}$, Feng-Yu Wang$^{(a)}$ and Simon Wittmann$^{(c,d)}$\footnote{Corresponding author.}\\
	\footnotesize{$a)$  Center for Applied Mathematics, Tianjin
		University, Tianjin 300072, China}\\
	\footnotesize{$b)$ Department of Mathematics, City University of  Hong Kong, Hong Kong,  China}\\
	\footnotesize{$c)$ Department of Mathematics, The Hong Kong Polytechnic University,  Hong Kong,  China }\\
	\footnotesize{$d)$ Faculty of Mathematics, Wroclaw University of Science and Technology,   50-370 Wroclaw,  Poland }\\
	\footnotesize{ panparen@cityu.edu.hk,    wangfy@tju.edu.cn,  simon.wittmann@pwr.edu.pl}}

\maketitle

\vspace{-0.8cm}

\begin{abstract}
	To study diffusion processes on the $p$-Wasserstein space $\mathscr P_p$ for $p\in [1,\infty)$ over a separable, reflexive Banach space $X$, we present a criterion on the quasi-regularity of  Dirichlet forms in $L^2(\mathscr P_p,\Lambda)$ for a reference probability $\Lambda$ on $\mathscr P_p$.  It is formulated in terms of an upper bound condition with   the uniform norm of the intrinsic derivative. We find a versatile class of quasi-regular local Dirichlet forms on $\mathscr P_p$ by using images of Dirichlet forms on the tangent space $L^p(X\to X,\mu_0)$ at a reference point $\mu_0\in\mathscr P_p$. The Ornstein--Uhlenbeck type Dirichlet form and process on the $2$-Wasserstein space over a separable Hilbert space are an important example in this class. We derive an $L^2$-estimate for the corresponding heat kernel and an integration by parts formula for the invariant measure.
\end{abstract}


\section{Introduction}\label{sec:const}

\noindent As a crucial topic in the crossed field of  probability theory, optimal transport and partial differential equations,  stochastic analysis on the Wasserstein space
has received much attention. Some   measure-valued  diffusion processes  have been constructed by using the theory of Dirichlet forms, see  \cite{KLV, ORS, Sch, Shao, Sturm, Sturm24} and the references therein.
The pre-Dirichlet forms are defined  by integrating a square field operator with respect to a reference Borel probability measure $\LL$ on a topological space, whose points are measures over a Riemannian manifold, or $\R^d$. The square field operators are determined  by the intrinsic or extrinsic derivatives, which describe  the stochastic motion and birth-death of particles respectively.
The selection of reference measures $\LL$ found in the literature  is typically supported on the class of singular measures. Hence, these do not provide natural options when looking for a suitable substitute for a volume measure or a Gaussian measure on the set of probability measures.
On the other hand, for stochastic analysis on the Wasserstein space, it is essential to construct a diffusion process which plays the role of Brownian motion in finite dimensions,
or the Ornstein--Uhlenbeck (O-U for short) process on a separable Hilbert space.
This has been a long standing open problem due to the lack of a volume or Gaussian measure on such a state space, which could serve as an invariant measure. 
As a solution to this problem, \cite{RW22} presents a general technique to construct an abundance of Gaussian-like  probability measures on $\scr P_2(\R^d)$ together with the related O-U type Dirichlet forms.  The construction  is very natural as the Gaussian measure and the related Dirichlet form are obtained as images of the corresponding objects from the tangent space $T_{\mu_0,2}:=L^2(\R^d\to\R^d,\mu_0)$ at a fixed element $\mu_0\in \scr P_2(\mathbb R^d)$ which is absolutely continuous w.r.t.~the Lebesgue measure on $\R^d$.

The quasi-regularity of a Dirichlet form is a necessary and sufficient condition for the existence of an associated right process (see \cite[Chap.~IV]{MR92}).
However, concerning the O-U type Dirichlet form on $\scr P_2(\R^d)$ the question of quasi-regularity   is open until now.
In this article we prove a criterion for quasi-regularity which only requires an upper bound of the bilinear form on a pre-domain of differentiable cylinder functions (Theorem \ref{QRT} below). 
In more generality, we prove the existence of a versatile class of diffusion processes on the $p$-Wasserstein space over a separable, reflexive Banach space $X$.
The existence of a Markov process on $\scr P_2(\R^d)$ associated with the O-U type Dirichlet form of \cite{RW22} is an important application
of this  result.
Although the O-U type process serves as the principle example in this article,
the existence result comprises a broader class. The Dirichlet forms are assumed to be of gradient type.
In addition, regarding Gaussian-like measures on the $2$-Wasserstein space  over a separable Hilbert space,
we derive an $L^2$-estimate for the corresponding heat kernel and an integration by parts formula.

In the next paragraph  the main idea behind the construction of an O-U type process on the $2$-Wasserstein space is outlined.
Here, we consider the state space $\scr P_2(\mathbb R^d)$ for simplicity. In Section \ref{sec:applic} of this article the Gaussian-like measure and the O-U type process on the $2$-Wasserstein space over a separable Hilbert space is introduced.

The starting point is a fact  from the theory of optimal transport, which can be found in \cite{V09} or \cite{AGS05}, for example:
The map
\begin{equation*}
\Psi: T_{\mu_0,2} \ni h\mapsto \mu_0\circ h^{-1}\in \scr P_2(\mathbb R^d)
\end{equation*}
is surjective.
Moreover, $\Psi$ is $1$-Lipschitz continuous w.r.t.~the $2$-Wasserstein distance
$$\W_2(\mu,\nu):=\inf_{\pi\in \C(\mu,\nu)} \bigg(\int_{\R^d\times \R^d} |x-y|^2\,\d\pi(x,y)\bigg)^{\ff 1 2},\quad\mu,\nu\in \scr P_2(\mathbb R^d),$$
where $\C(\mu,\nu)$ denotes the set of all couplings of $\mu$ and $\nu$.
Now, let $G$ be a non-degenerate Gaussian measure on the Hilbert space $T_{\mu_0,2}$ with trace-class
covariance operator $A^{-1}$, where $(A,\D(A))$ is a positive definite self-adjoint operator in $T_{\mu_0,2}$. The associated   O-U  process on $T_{\mu_0,2}$ is generated by
\begin{equation}\label{eqn:OUgen}
	L^{\text{OU}}u(h):=\Delta u(h)-\langle A\nabla u(h),h\rangle_{T_{\mu_0,2}},\quad h\in T_{\mu_0,2},\, u\in \D(L^{\text{OU}})\subseteq L^2(T_{\mu_0,2},G).
\end{equation} Here, $\nn$ and $\DD$ denote the gradient and Laplacian on  $T_{\mu_0,2}$ respectively.
The O-U process $(X_t)_{t\ge 0}$ on the tangent space can be constructed as the mild solution of the corresponding semi-linear SPDE, i.e.
\begin{equation*}
X_t= \e^{-At} X_0+\ss 2 \int_0^t\e^{-(t-s)A}\d W_s,\ \ t\ge 0,
\end{equation*}
where $W_t$ is the standard cylindrical Brownian motion on $T_{\mu_0,2}$ (see e.g.~\cite[Chap.~6]{DZ}).
The associated O-U Dirichlet form $(\tt\EE,\D(\tt\EE))$ is known to be the closure of
$$\tt\EE(f,g):= \int_{T_{\mu_0,2}}\<\nn f(\phi),\nn g(\phi)\>_{T_{\mu_0,2}}\d G(\phi),\qquad f,g\in C_b^1(T_{\mu_0,2}).$$
Now, under the map $\Psi$, the image of the Gaussian measure $G$ yields a reference measure
\begin{equation*}
N_G:= G\circ\Psi^{-1}
\end{equation*} on $\scr P_2(\mathbb R^d)$, which we   call
the Gaussian measure   induced by $G$.
The $\Psi$-image $(\EE,\D(\EE))$ of   $(\tt\EE,\D(\tt\EE))$ is a symmetric conservative local Dirichlet form in $L^2(\scr P_2(\R^d),N_G)$
satisfying
\beq\label{WP0}
\EE(u,v)=\int_{\scr P_2(\mathbb R^d)} \<Du(\mu), Dv(\mu)\>_{L^2(\R^d\to\R^d,\mu)}\d N_G(\mu),\qquad u,v\in C_b^1(\scr P_2(\mathbb R^d)),
\end{equation} where $D$ is the intrinsic derivative on $\scr P_2(\mathbb R^d)$. It is first introduced in \cite{AKR}
on the configuration space over a Riemannian manifold, and a transfer of this notion (see \cite{BRW} and Definition \ref{D}  below) 
allows us to define  the class $C_b^1(\scr P_p(X))$ for $p\in [1,\infty)$ and a Banach space $X$.
The Dirichlet form $\EE$ in \eqref{WP0} has an analogous shape  as the O-U Dirichlet form on a Hilbert space. Moreover, as shown in \cite{RW22},  it inherits  several nice properties  from $ \tt\EE$,
including the log-Sobolev inequality and compactness of its semigroup. 
However, the question of quasi-regularity and hence the existence of a Markov process on $\scr P_2(\R^d)$ associated with $\EE$ is still open up to now.

It is worth pointing out that the choice of $\mu_0$  does not matter in this regard,
since the set
\begin{equation*}
\big\{N_G=G\circ\Psi^{-1}:G\text{ is a Gaussian on } T_{\mu_0,2}\big\}
\end{equation*}
is invariant under a change of the reference point $\mu_0$, as long as $\mu_0$ is absolutely continuous.
In the setting of $\scr P_2(\R^d)$ this condition on $\mu_0$ means being absolute continuous w.r.t.~the Lebesgue measure.
In the general   setting of $\scr P_2(H)$ for a separable Hilbert space $H$, addressed in Section \ref{sec:applic}, the condition means being absolutely continuous w.r.t.~a non-degenerate Gaussian measure on $H$.
  Hence, the family of Gaussian-like measures and corresponding O-U type processes  in any of these cases do not depend on $\mu_0$ (see Remark \ref{rem:inv} below).
The generator of $\EE$  can be formally represented as the intrinsic Laplacian with a drift (see \cite{RW22}). The reference measure $N_G$ satisfies an integration by parts formula (see Sect \ref{sec:ibpf} below) in
analogy to the corresponding formula for the Gaussian $G$.
So, we refer to  $(\EE,\D(\EE))$ as an O-U type Dirichlet form on the $2$-Wasserstein space.

Since quasi-regularity is the key to construct Markov processes using  Dirichlet forms, the existence of the corresponding O-U type stochastic process is still an open problem up to now.
As such a process is of wide interest, we prove a handy, general criterion (see Theorem \ref{QRT} below) for the quasi-regularity of Dirichlet forms on the Wasserstein space over a Banach space $X$, 
which in particular applies to the above example, but also in  more general situations.
The framework in this article allows us to construct diffusion processes on the
$p$-Wasserstein space $\scr P_p(X)$ over a separable, reflexive Banach space $X$ for $p\in [1,\infty)$. Application to the O-U type process for $p=2$ in case $X$ is a Hilbert space  serves as a typical and highly relevant example. 
The proof of quasi-regularity is inspired by  the methods developed in \cite{RS92} and \cite{RS95}.
The latter of these two only presumes a Polish state space.
Nevertheless, application of the techniques and verification of conditions require a detailed analysis which takes into account the special nature and properties of the metric and the topology involved. 
An adaptation  to the $p$-Wasserstein distance, as realized in this article, is completely new.
For the configuration space equipped with the vague topology, a similar result has been achieved in \cite{MR00}.
Regarding the weak topology on the set of Borel probability measures over a Polish space, \cite{ORS} provides a quasi-regularity result. The latter, however, focuses on Dirichlet forms linked to the extrinsic derivative instead of the intrinsic.  Our main results, Theorems \ref{QRT} \& \ref{TN}, relating to the intrinsic derivative and the $p$-Wasserstein distance,
show quasi-regularity for a wide class of Dirichlet forms with state space $\scr P_p(X)$.
The methods of this paper should also be applicable for the Wasserstein space over non-linear metric spaces like Riemannian manifolds. To save space  we leave this for a future study.

Throughout this text, let $(X,\|\cdot\|_X)$ be a separable, reflexive Banach space.
For short, we write $\scr P$ instead of $\scr P(X)$ to denote the space of probability measures on $X$ and accordingly
write $\scr P_p$ instead of $\scr P_p(X)$
for  $p\in [1,\infty)$,  i.e.
\beq\label{PP}\scr P_p:=\big\{\mu\in \scr P:\ \mu(\|\cdot\|_X^p)<\infty\big\}.\end{equation}
As stated in \cite[Thm.~6.18]{V09}, the $p$-Wasserstein distance
\beq\label{WP}\W_p(\mu,\nu):=\inf_{\pi\in \C(\mu,\nu)} \bigg(\int_{X\times X} \|x-y\|_X^p\,\d\pi(x,y)\bigg)^{\ff 1 p},\quad\mu,\nu\in \scr P_p,\end{equation} 
yields a complete, separable metric on $\scr P_p$. 
It is  worth mentioning that the metric space $(\scr P_p,\W_p)$ is not locally compact, not even  in case $X=\R^d$. 
The subsequent list sums up our main results.
\begin{itemize}
	\item We provide a general quasi-regularity condition for a Dirichlet form $(\EE,\D(\EE))$ in \sloppy $L^2(\scr P_p,\LL)$ for a given reference probability $\LL$ on $\scr P_p$ under the following condition:
	If $\D(\EE)$ has a dense subset of quasi-continuous functions and moreover for all differentiable functions $f:\scr P_p\to\R$ of cylindrical type
	satisfies the  inequality
	\begin{equation*}
		\EE(f,f)\leq C\sup_{\mu\in\scr P_p}\|Df(\mu)\|^2_{L^{p^*}(X\to X^*,\mu)}
	\end{equation*}
	for some constant $C\in(0,\infty)$ and $p^*:=\ff p{p-1}\in(1,\infty]$, then $(\EE,\D(\EE))$ is quasi-regular on $\scr P_p$.
	\item Fixing $\mu_0\in\scr P_p$ together with a probability $\LL_0$  on $T_{\mu_0,p}:=L^p(X\to X,\mu_0)$
	 and a Dirichlet form $(\tt\EE,\D(\tt\EE))$ in $L^2(T_{\mu_0,p},\LL_0)$ we obtain a suitable measure $\LL$ on $\scr P_p$
	 and a quasi-regular $(\EE,\D(\EE))$ via the push-forward under the map
	 \begin{equation*}
	 	\Psi:T_{\mu_0,p}\ni\phi\mapsto \mu_0\circ\phi^{-1}\in \scr P_p.
	 \end{equation*}
	 In case $p\in[1,2]$ and $X$ is a Hilbert space: If $\tt\EE$ is a diffusion form, then $\EE$ is as well, i.e.~a local Dirichlet form with square-field operator
	 \begin{equation*}
	 	\Gamma(u,v)(\mu)=\langle D u(\mu),B_\mu( D v(\mu))\rangle_{L^2(X\to X,\mu)},
	 \end{equation*}
	 where $B_\mu:T_{\mu,2}\to T_{\mu,2}$ is a field  of bounded, symmetric operators over $\mu\in\scr P_p$,
	 and $\EE$ is associated to a conservative diffusion process on $\scr P_p$.
	\item For the choices  $p=2$,  $X=H$ as a separable Hilbert space and $\LL_0:=G$ as a non-degenerate Gaussian on $L^2(X\to X,\mu_0)$, we 
	obtain a Markov process on $\scr P_2$ which has invariant measure $N_G:=G\circ\Psi^{-1}$ and satisfies $L^2$-bounds for its transition kernel ${( p_t)}_{t\ge 0}$ in terms of the eigenvalues
	${\{\alpha_n\}}_{n\in\N}$ of the covariance operator of $G$,
	\beg{equation*} \int_{\scr P_2\times \scr P_2} p_t(\mu,\nu)^2 \,\d N_G(\mu) \,\d N_G(\nu)
	\le  \prod_{n\in \mathbb N}  \Big(1+ \ff{2\e^{-2\alpha_n t}}{(2\alpha_n t)\land 1}\Big)<\infty,\ \ t>0.\end{equation*}
	Moreover, an integration by parts formula for $N_G$ is obtained.
\end{itemize}

The outline of this article is as follows:
In Section \ref{sec:QR}  we address the property of quasi-regularity for Dirichlet forms in   $L^2(\scr P_p,\LL)$ and develop a mild criterion which involves the uniform norm of the intrinsic derivative. In Section \ref{sec:gradForm} we apply this  criterion   to  construct  a class of quasi-regular local Dirichlet forms and diffusion processes. These are obtained as images of Dirichlet forms on the tangent space at a fixed point of the Wasserstein space. In Section \ref{sec:applic} we confirm the quasi-regularity of the O-U type Dirichlet form and  give an $L^2$-estimate for its heat kernel as well as an integration by parts formula for its invariant measure.

\section{Quasi-regular Dirichlet forms on \texorpdfstring{$\scr P_p$}{Pp}}\label{sec:QR}

\noindent We first recall some notions on Dirichlet forms which can be found in  \cite{MR92}.

Let $(E,\rr)$ be a Polish space   and $\LL$ be a probability measure on the Borel $\si$-algebra $\scr B(E)$.
A Dirichlet form  $(\EE,\D(\EE))$ on $L^2(E,\LL)$ is a densely defined, closed bilinear form which is Markovian.
As a convention in this text, to have a simpler and more convenient notation, a measurable function $f$ on $(E,\scr B(E))$ is identified with its $\LL$-equivalence class of measurable functions.
For example, we write $f\in L^2(E,\LL)$ or $f\in\D(\EE)$ if its class has the respective property.
We write $\LL(f):=\int f\d\Lambda$ if the integral exists.
In this article, only Dirichlet forms which are symmetric and conservative  play a role, i.e.~$\EE(f,g)=\EE(g,f)$ for $f,g\in\D(\EE)$ and $\eins_E\in\D(\EE)$ with $\EE(\eins_E,\eins_E)=0$.
So, by convention, if we refer to a generic Dirichlet form in an abstract context, then symmetry and conservativeness are always assumed. 
We set
$$\EE_1(f,g):=\Lambda(fg)+\EE(f,g),\quad f,g\in\D(\EE).$$ 
For an open set $O\subseteq E$ the $1$-Capacity associated to $\EE$ is defined as
\begin{equation*}
{\rm Cap}_1(O):=\inf \big\{\EE_1 (f,f)\,:\, f\in \D(\EE),\, f(z)\ge 1\text{ for }\LL\text{-a.e.~}z\in O\big\}
\end{equation*}
with the convention of $\inf(\emptyset):=\infty$. For an arbitrary set $A\subseteq E$, let
$${\rm Cap}_1(A):= \inf \big\{{\rm Cap}_1(O)\,:\, A\subseteq O,\, O \, \text{is\ an open set in }E\big\}.   $$

An $\EE$-nest  (or nest for short) is a sequence of closed subsets $\{K_n\}_{n\in\N}$ of $E$ such that
$$\lim_{n\to\infty}{\rm Cap}_1(E\setminus K_n)= 0.$$

A measurable function $f: E\to\R$ is called quasi-continuous  if there exists a nest $\{K_n\}_{n\in\N}$ such that the restriction $f|_{K_n}$ is continuous for each $n\in\N$.
A sequence $\{f_k\}_{k\in\N}$ of measurable functions   is said to converge quasi-uniformly to a function $f:E\to\R$  if there exists a nest $\{K_n\}_{n\in\N}$ such that the sequence of restricted functions $f_k|_{K_n}$, $k\in\N$, converge to $f|_{K_n}$ uniformly on $K_n$ as $k\to\infty$ for each $n\in\N$.
If a property, which an element $z\in E$ either has or doesn't, holds for all $z$ in the complement of a set $N\subseteq E$ with $\textnormal{Cap}_1(N)=0$, then this property is said to hold quasi-everywhere (q.-e.) on $E$.

The following Definition is equivalent to \cite[Def.~IV.3.1]{MR92} in the conservative case.
\beg{defn}\label{DF2} The  Dirichlet form $(\EE,\D(\EE))$    is called quasi-regular if the following three conditions are met.
\beg{enumerate}
\item[1)]  There exists an  $\EE$-nest of compact sets (i.e.~$\textnormal{Cap}_1(\,\cdot\,)$ is tight).
\item[2)]  The Hilbert space  $(\D(\EE),\EE_1)$ has a dense subspace consisting of quasi-continuous functions.
\item[3)] There exists $N\subseteq E$ with $\textnormal{Cap}_1(N)=0$  and a
sequence $\{f_i\}_{i\in\N} \subseteq \D(\EE)$ of   quasi-continuous functions  which
separate points in $E\setminus N$, i.e.~for any two different points $z_1,z_2\in E\setminus N$ there exists $i\in\N$ such that $f_i(z_1)\neq f_i(z_2)$.
\end{enumerate}\end{defn}

\subsection{A criterion of the quasi-regularity}

\noindent From now on, we fix $p\in [1,\infty)$
and consider $(E,\rr)=  (\scr P_p,\W_p)$ given as in \eqref{PP} and \eqref{WP}.  To prove the quasi-regularity of a Dirichlet form $(\EE,\D(\EE))$ in $L^2(\scr P_p,\LL)$,  we look at the class   $C_b^1(\scr P_p)$ defined as follows by using the intrinsic derivative. This derivative   is first introduced in \cite{AKR} on the configuration space over a Riemannian manifold, and has been extended in \cite{BRW} to $\scr P_p$ over a Banach space.

Let $X^*$ be the dual space of $X$, i.e.~$X^*$ is the Banach space of bounded linear functionals $X\to\R$. We write
$$\phantom{}_{X^*}\<x',x\>_X:= x'(x),\ \ x'\in X^*,\ x\in  X.$$
Let $p^*=\ff p{p-1}$ (which yields $\infty$ if $p=1$). The   tangent space of $\scr P_p$ at a point $\mu\in \scr P_p$ is
defined as $$T_{\mu,p}:=L^p(X\to X,\mu).$$
By virtue of \cite[Thm.~IV 1.1 \& Cor.~III 3.4]{Diestel}, its dual space $T_{\mu,p}^*$ can be identified with the space
$L^{p^*}(X\to X^*,\mu)$. Accordingly, we write
$$\phantom{}_{T_{\mu,p}^*}\<\phi',\phi\>_{T_{\mu,p}}:=\int_{X} \phantom{}_{X^*}\<\phi'(x),\phi(x)\>_X\,\mu(\d x),\ \ \phi'\in T_{\mu,p}^*,\ \phi\in T_{\mu,p}.$$

Let $id$ denote the identity function $X\to X$.
\beg{defn}\label{D} Let $f$ be a continuous function on $\scr P_p$.
\begin{itemize}
\item $f$ is called intrinsically differentiable
if for every $\mu\in \scr P_p$
$$T_{\mu,p}\ni \phi\mapsto D_\phi f(\mu):=\lim_{\vv\to 0}\ff{f(\mu\circ(id+\vv\phi)^{-1})-f(\mu)}\vv $$
is a bounded linear functional. In this case the intrinsic derivative of $f$ at $\mu$ is defined as the unique element $Df(\mu)\in T_{\mu,p}^* $ such that
$$D_\phi f(\mu)=\phantom{}_{T_{\mu,p}^*}\<Df(\mu), \phi\>_{T_{\mu,p}}:={\int_X}\phantom{}_{X^*}\<Df(\mu)(x), \phi(x)\>_X\,\mu(\d x),\ \ \phi\in T_{\mu,p}.$$
\item We write $f\in C^1(\scr P_p)$ if $f$ is intrinsically differentiable such that
$$\lim_{\|\phi\|_{T_{\mu,p}}\downarrow 0}\ff{|f(\mu\circ(id+\phi)^{-1})-f(\mu)-D_\phi f(\mu)|}{\|\phi\|_{T_{\mu,p}}}=0,\ \  \mu\in \scr P_p,$$
and  $Df(\mu)(x)$ has a continuous version in $(\mu,x)\in \scr P_p\times X,$ i.e. there exists a continuous map $g:\scr P_p\times X\to X^*$ such that $g(\mu,\,\cdot\,)$ is a $\mu$-version of $Df(\mu)$ for each $\mu\in\scr P_p$. In this case we always take $Df$ to be  its continuous version, which is unique.
\item We write $f\in C^1_b(\scr P_p)$ if $f\in C^1(\scr P_p)$ and   $|f|+\|Df\|_{X^*}$ is bounded on $\scr P_p\times X.$\end{itemize}
\end{defn}

A typical subspace of  $C_b^1(\scr P_p)$ is the class of   cylindrical functions   introduced as follows.
First, we recall the continuously differentiable cylindrical functions on the Banach space $X$,
\begin{equation*}
\scr F C_b^1(X) := \big\{  g(x_1',\cdots, x_n'):\,\ n\in\N,\, x_i'\in X^*,\, g\in C_b^1(\R^n)\big\}.
\end{equation*}
It is well known that each
\begin{equation*}
\psi:= g(x_1',\cdots, x_n') \in \scr F C_b^1(X) 
\end{equation*} belongs to $C_b^1(X)$, since $\psi$ is bounded and Fr\'echet differentiable on $X$ with bounded and continuous derivative
\begin{equation*}
\nabla \psi(x)= \sum_{i=1}^n (\pp_i g) \big( \phantom{}_{X^*}{\langle}x_1',x \rangle_X,\cdots, \phantom{}_{X^*}{\langle}x_n',x\rangle_X\big)x_i',\quad \ x\in X.
\end{equation*}

Next, we consider the space of continuously differentiable cylindrical functions,
\beq\label{CF}\scr F C_b^1(\scr P) := \big\{\scr P\ni \mu\mapsto g(\mu(\psi_1),\cdots, \mu(\psi_n))\,: \,\ n\in\N,\, \psi_i\in \scr FC_b^1(X),\, g\in C_b^1(\R^n)\big\}.\end{equation}
For a function $f\in \scr F C_b^1(\scr P)$ with  $f(\mu):=g(\mu(\psi_1),\cdots, \mu(\psi_n))$ the map $\scr P_p\ni \mu\mapsto f(\mu)\in\R$ is easily seen to be an element of $C_b^1(\scr P_p)$ with
\begin{equation}\label{eqn:Df}
D f(\mu)(x)= \sum_{i=1}^n (\pp_i g) (\mu(\psi_1),\cdots, \mu(\psi_n)) \nn \psi_i(x)\in X^*,
\quad  (\mu,x)\in\scr P_p\times X.
\end{equation}
In the following, we  write $\scr F C_b^1(\scr P_p)$  for the restrictions of  functions in  $\scr F C_b^1(\scr P)$   to $\scr P_p$.
Then, $\scr F C_b^1(\scr P_p)$ is dense in $L^2(\scr P_p,\LL)$ for any probability measure $\LL$ on $\scr P_p$ (Lemma \ref{L2} below).

\beg{thm}\label{QRT} A Dirichlet form   $(\EE,\D(\EE))$    in $L^2(\scr P_p,\LL)$ is quasi-regular if it satisfies the following two conditions:
\begin{itemize}
\item[$(C_1)$]   $ \scr FC_b^1(\scr P_p)\subseteq \D(\EE)$  and there exists a constant $C\in (0,\infty)$ such that
\begin{equation*}
\EE(f,f)\leq C\sup_{\mu\in\scr P_p}\|Df(\mu)\|^2_{T_{\mu,p}^*},\ \ f\in \scr FC_b^1(\scr P_p).
%		\sup_{\mu\in\scr P_p}\,\int_X \|D\tilde u(\mu)(x)\|^r_{X^*}\,\d\mu(x).
\end{equation*}
\item[$(C_2)$] The Hilbert space  $(\D(\EE),\EE_1)$ has a dense subspace consisting of  quasi-continuous functions.
\end{itemize}
\end{thm}

We would like to indicate that the conditions in Theorem \ref{QRT} are easy to check in applications.
If the Dirichlet form $(\EE,\D(\EE))$ is  constructed as the closure (i.e.~the smallest closed extension, see e.g.~\cite[Chap.~I]{MR92}) of a closable  bilinear form defined on  $C_b^1(\scr P_p)$ or $ \scr FC_b^1(\scr P_p)$, then the second condition holds automatically. The first condition holds if    there exists a positive function $F\in L^1(\scr P_p,\LL)$ such that
$$\EE(f,f)\le \int_{\scr P_p} F(\mu) \|Df(\mu)\|^2_{T_{\mu,p}^*}\LL(\d\mu),\ \ f\in \scr FC_b^1(\scr P_p).$$
For example, $F(\cdot)$ may be a dominating function for some diffusion coefficients, see Section \ref{diffusion} below.

\subsection{Proof of the criterion}\label{sec:prelimi}

\noindent We first present some lemmas.

\begin{lem}\label{lem:onBorel}
There exists  a sequence ${\{\psi_n\}}_{n\in\N}\subseteq  \scr FC_b^1(X)$ with the following properties.
\begin{enumerate}
\item[$(i)$]  The family $ \{\scr P\ni \mu\mapsto\mu(\psi_n)\}_{n\in\N} $ separates the points on $\scr P$, and hence  separates the points on $\scr P_p$.
\item[$(ii)$] The $\sigma$-algebra   generated by the family $ \{\scr P\ni \mu\mapsto\mu(\psi_n)\}_{n\in\N} $ coincides with the Borel $\sigma$-algebra $\scr B(\scr P_p)$.
\end{enumerate}
\end{lem}
\begin{proof} Since $X$ is a separable Banach space,
there exists a sequence  ${\{x'_i\}}_{i\in\N}$ separating the points on $X$. So, 
\begin{equation}\label{eq:family}
\Big\{X\ni x\mapsto \prod_{j=1}^m\frac{x'_{i_j}(x)}{1+|x'_{i_j}(x)|}\in[0,1]\;\;:\;\; m\in\N,\,i_1,\dots, i_m\in\N\Big\}
\end{equation}
is a countable subset of $\scr FC_b^1(X)$, closed under multiplication and separating the points on $X$.
We arrange the elements of the family \eqref{eq:family} into a sequence ${\{\psi_n\}}_{n\in\N}$.
It satisfies $(i)$ according to   \cite[Theorem~11(b)]{BK10}.

To verify $(ii)$,  let $\bar \si$ be the $\si$-algebra on $\scr P_p$ induced by $\{\mu\mapsto \mu(\psi_n)\}_{n\in\N}$, and let    $\si(\tau_\rr)$ denote the $\sigma$-algebra generated by the family of all open sets $\tau_\rho$ w.r.t.~the metric
\begin{equation*}
\rho(\mu,\nu):=\sum_{n\in\N} 2^{-n}  |\mu(\psi_n)-\nu(\psi_n)|,
\quad \mu,\nu\in\scr P_p.
\end{equation*}
Then, $\si(\tau_\rr)\subseteq\bar\si.$
Noting that   $\{\psi_n\}_{n\in\N}$ are continuous and uniformly bounded,	according to   \cite[Lemma 3(a)]{BK10},  property $(i)$ of this lemma implies  $\scr B(\scr P_p)=\si(\tau_\rr),$
so that $\scr B(\scr P_p)\subseteq \bar\si.$ On the other hand, each $\psi_n$ is a bounded continuous function on $X$, so that $\mu\mapsto \mu(\psi_n)$ is continuous in $\scr P_p$, and hence
$\bar\si\subseteq  \scr B(\scr P_p).$ Therefore,  $(ii)$ is satisfied.
\end{proof}

\begin{lem}\label{L2}
The linear space $\scr FC_b^1(\scr P_p)$ is dense
in $L^2(\scr P_p,\LL)$.
\end{lem}
\begin{proof}
Let $\scr A$ be the class of all subsets $A\subseteq \scr P_p$ given by
\begin{align*}
A:=\bigcap_{i=1}^m\,\{\mu\in\scr P_p\,:\,\mu(\psi_i)\in B_i\}
\end{align*}
for some $m\in\N$, $B_1,\dots B_m\in\scr B(\R)$ and $\psi_1,\dots,\psi_m\in\scr FC_b^1(X).$
Obviously, $\scr A$ is   $\cap\,$-stable, i.e.~$A_1\cap A_2\in\scr A$ for $A_1,A_2\in \scr A$.
Furthermore, Lemma \ref{lem:onBorel}(ii) implies that $\sigma(\scr A)=\scr B(\scr P_p)$.
Now, let $V$ denote the vector space of real-valued, measurable functions on $\scr P_p$ which coincide in $\LL$-a.e.~sense with some element of $\overline{\scr C^\textnormal{cyl}}$,
the topological closure of $\scr FC_b^1(\scr P_p)$ in $L^2(\scr P_p,\LL)$.
The claim of this Lemma reads $\overline{\scr C^\textnormal{cyl}}=L^2(\scr P_p,\LL)$.
It suffices to show that $V$ contains every bounded, measurable function.
The monotone class theorem for functions (see \cite[Preliminaries, Theorem 2.3]{BG68})
yields exactly the desired statement if $V$ meets the following properties:
\begin{enumerate}
\item[1)] $V$ contains the indicator function $\eins_A$ for any set $A\in\scr A$.
\item[2)] If $\{f_n\}_{n\in\N}\subseteq V$ is an increasing sequence of non-negative functions such that $f:=\lim_{n\to\infty}f_n$ is bounded  on $\scr P_p$, then $f\in V$.
\end{enumerate}

To show 1), let  $A\in\scr A$, and $m\in\mathbb N$, $B_1,\dots,B_m\in \B(\R)$ and $\psi_1,\dots,\psi_m\in \scr F C_b^1(X)$ such that
$$\eins_A(\mu)=\prod_{i=1}^m\eins_{B_i}(\mu(\psi_i)),\ \ \mu\in\scr P_p.$$ If we denote the image measure of $\LL$ under $\scr P_p\ni \mu\mapsto(\mu(\psi_1),\dots,\mu(\psi_m))\in\R^m$ by $\LL_m$, then
$$ \int_{\scr P_p} \big| g(\mu(\psi_1),\dots,\mu(\psi_m))-\eins_A\big|^2  \, \d\LL(\mu) = \int_{\R^m} \Big| g(x)-\prod_{i=1}^m\eins_{B_i}(x_i)\Big|^2 \,\d \LL_m( x)$$
holds for any $g\in C_b^1(\R^m)$.
This implies  $\eins_A\in V$, since   $C_b^1(\R^m)$ is dense in $L^2(\R^m,\LL_m)$,

Now, let    $\{f_n\}_{n\in\N}\subseteq V$ and $f$ be as in 2). It remains to show  that  $f\in  V$.
For $\varepsilon>0$ there exists $n\in\N$ such that
$(\int_{\scr P_p}|f-f_n|^2\d\LL)^{1/2}\leq \varepsilon/2$. Since $f_n\in V$,
we find  $u_n\in \scr FC_b^1(\scr P_p)$ such that
$(\int_{\scr P_p}| f_n-u_n|^2\d\LL)^{1/2}\leq \varepsilon/2$. Then the Minkowski inequality yields
$(\int_{\scr P_p}|u_n-f|^2\d\LL)^{1/2}\leq \varepsilon$. Hence, $f\in V$ as desired.
\end{proof}

To verify the tightness of capacity, we shall construct a class of reference functions.
For $a, b\in \R$ let $a\lor b:=\max\{a,b\}$,
$a\land b:=\min\{a,b\}$ and $a^+:=a\lor 0$.
For any $l\in \N:=\{1,2,\dots\}$, let
\begin{equation}\label{eqn:chi}
\chi_l(s):=-\frac{3l}{2} +\int_{-\infty}^s\Big[\Big(\frac{t}{l}+2\Big)^+\land 1\Big]
\land\Big[\Big(2-\frac{t}{l}\Big)^+\land 1\Big]\,\d t,\quad t\in\R.
\end{equation}
Then $\chi_l\in C_b^1(\R)$   with
\begin{equation*}
\chi_l(s)=s\quad\text{for } s\in[-l,l],\quad\text{and}\quad
\eins_{[-l,l]}(s)\leq\chi_l'(s)\leq\eins_{[-2l,2l]}(s),\quad s\in\R.
\end{equation*}
We recall the notation $p^*:=\frac{p}{p-1}\in(1,\infty]$ for $p\in[1,\infty)$ as introduced above.

\begin{lem}\label{lem:abc} Assume that Condition $(C_1)$ in  Theorem $\ref{QRT}$ holds.
Let $\Phi\in C_b^1(\R)$ and $\gamma\in C^1(\R,[0,\infty))$ such that for some  constants $a,b\in(0,\infty)$   the estimates
$$\sup_{s\in[0,\infty)}|\Phi'(s)|(1+s)^{\ff 1 {p^*}}\leq a\quad\text{and}\quad\sup_{s\in\R} |\gamma'(s)| (1+\gamma(s))^{-\ff 1 {p^*}}\le b$$
are satisfied.
Then,  for any $m\in\N$, a  Lipschitz function $f$ on $\R^m$,   and $x_i'\in X^*$ with $\|x_i'\|_{X^*}=1$ for $i=1,\dots,m$,
the function
\begin{equation*}
u(\mu):=\Phi\big(\mu\big(\gamma\circ f (x_1',\dots,x_m')\big)\big)\quad \text{for } \mu\in\scr P_p
\end{equation*}
belongs to $\D(\EE)$ with
\begin{equation*}
\EE(u,u)\leq C(ab)^2 \Big\|\sum_{i=1}^m|\partial_i f|\Big\|_{L^\infty(\R^m)}^2.
\end{equation*}
\end{lem}	

\begin{proof}
The proof is split into three parts. In (a) we prove the claim under the additional assumption that $f$ is a bounded, continuously differentiable function.
In (b) we prove the claim under the assumption that $f\in C^1(\R^m)$. In the final part (c) we assume that $f$ has a global Lipschitz constant as in the statement of the lemma.

(a) We first prove the result under the additional assumption that $f\in C_b^1(\R^m)$. In this case,
$\gamma\circ f\in C_b^1(\R^m)$ and the function
\begin{equation*}
u(\mu):=\Phi\big(\mu\big(\gamma\circ f (x_1',\dots,x_m')\big)\big),\quad \mu\in\scr P_p,
\end{equation*}
is in  $u\in\scr FC_b^1(\scr P_p)$.
By \eqref{eqn:Df}, with
\begin{equation*}
T:X\ni x\mapsto\big(\phantom{}_{X^*}{\langle}x_1',x{\rangle}_X ,\dots,
\phantom{}_{X^*}{\langle}x_m',x{\rangle}_X\big)\in\R^m,
\end{equation*}
it holds
\begin{equation*}
Du(\mu)(x)=\Phi'\big(\mu(\gamma\circ f \circ T)\big)
(\gamma'\circ f)(Tx)
\sum_{i=1}^m\partial_i f(Tx)\,x_i'
\end{equation*}
for $\mu\in\scr P_p$ and $x\in X$. So,
$$
\sup_{\mu\in\scr P_p}\|Du(\mu)\|_{T_{\mu,p}^*}^2\leq
(ab)^2 \Big\|\sum_{i=1}^m|\partial_i f|\Big\|_{L^\infty(\R^m)}^2,
$$
and the desired assertion follows from   Condition $(C_1)$ in  Theorem $\ref{QRT}$.

(b) Next, let $f\in C^1(\R^m)$.  For $l\in\N$ the composition
$f_l:=\chi_l\circ f$ yields an element of $C_b^1(\R^m)$ with
$|\partial_i f_l(z)|\leq |\partial_i f(z)|$ for $i=1,\dots, m$, $z\in\R^m$.
By what has been shown above,  for each $l$ the function
\begin{equation*}
u_l(\mu):=\Phi\big(\mu\big(\gamma\circ f_l (x_1',\dots,x_m')\big)\big)\quad \text{for } \mu\in\scr P_p
\end{equation*}
is in $\D(\EE)$ for $l\in\N$ with
\begin{equation}\label{eqn:eul}
\EE(u_l,u_l)\leq C (ab)^2\sup_{z\in\R^m}\Big( \sum_{i=1}^m|\partial_if(z)|\Big)^2.
\end{equation}
Clearly, $\lim_{l\to\infty}\gamma\circ f_l(Tx)=\gamma\circ f(Tx)$ for $x\in X$. By the condition on $\gg$
we find constants $A,B\in (0,\infty)$ such that
\beq\label{WL0}\big|\gamma\circ f_l(Tx)\big|\leq A\big(1+|f_l(Tx)|^p\big)\leq B\big(1+ \|x\|_X)^p\big),\ \ x\in X.\end{equation}
Then  Lebesgue's dominated convergence applies and
$$\lim_{l\to\infty}\mu(\gamma\circ f_l\circ T)=\mu(\gamma\circ f\circ T),\quad \mu\in \scr P_p.$$
Consequently, again by dominated convergence, we conclude $\lim_{l\to\infty}u_l=u$ in $L^2(\scr P_p,\LL)$.
Now, \cite[Lemma I.2.12]{MR92} yields $u\in\D(\EE)$ and $\EE(u,u)\leq\liminf_{l\to\infty}\EE(u_l,u_l)$, as the sequence $\{u_l\}_l$ is bounded w.r.t.~$\EE_1^{1/2}$-norm in view of \eqref{eqn:eul}.
Since \eqref{eqn:eul} delivers the desired upper bound, the proof of (b) is complete.

(c) Finally, let $f$ be a Lipschitz continuous function on $\R^m$. Then, $f$ is weakly differentiable on $\R^m$ with weak partial derivatives $\partial_i f\in L^\infty(\R^m)$. Let $h$ be a non-negative smooth function on $\R^m$ with compact support and
$\int h(z)\d z=1$. Then the mollifying approximations $\{f^{(n)}\}_{n\in \N}$ of $f$ defined by
$$f^{(n)}(x):= n^{-m} \int_{\R^m} f(z) h\big(n(z-x)\big)\d z= \int_{\R^m} f(x+n^{-1}z) h(z)\d z,\ \ x\in \R^m$$ satisfies $f^{(n)}\in C^1(\R^m)$,
$\lim_{n\to\infty} f^{(n)}=f$ and
$$\sup_{n\in\N}\sup_{z\in \R^m}\sum_{i=1}^m|\partial_if^{(n)} (z)|\le \Big\|\sum_{i=1}^m|\partial_if(z)|\Big\|_{L^\infty(\R^n)}.$$
Hence,
\beq\label{VV}\lim_{n\to\infty} \gamma\circ f^{(n)} (x_1',\dots,x_m')=\gamma\circ f  (x_1',\dots,x_m'),\end{equation}
and by step (b),
the functions
\begin{equation*}
u^{(n)}(\mu):=\Phi\big(\mu\big(\gamma\circ f^{(n)} (x_1',\dots,x_m')\big)\big),\ \  \mu\in\scr P_p,\ n\in\N,
\end{equation*}   belong to $\D(\EE)$ with
\beq\label{WL2}
\sup_{n\in\N}\EE(u^{(n)},u^{(n)})\leq C (ab)^2  \Big\|\sum_{i=1}^m|\partial_if(z)|\Big\|_{L^\infty(\R^m)}^2.
\end{equation} Moreover, since $f$ is Lipschitz continuous and $h$ is smooth with compact support,
we find constants $A,B\in (0,\infty)$ such that \eqref{WL0} holds. Thus,
by   the dominated convergence theorem, \eqref{VV} implies
$$\lim_{n\to\infty}\mu\big(\gamma\circ f^{(n)} (x_1',\dots,x_m')\big)= \mu\big(\gamma\circ f  (x_1',\dots,x_m')\big),\ \ \mu\in \scr P_p,$$ while  the definition  of $u$ and $u^{(n)}$ with $\Phi\in C_b^1(\R)$ yields
$$\lim_{n\to\infty}\|u^{(n)}-u\|_{L^2(\scr P_p,\LL)}=0.$$
According to \cite[Lemma I.2.12]{MR92},	this together with \eqref{WL2} finishes the proof.
\end{proof}

As an application    of Lemma \ref{lem:abc} we have the following assertion.

\begin{prp}	\label{prp:abc}
Let $\Phi\in C_b^1(\R)$ and $\gamma\in C^1(\R,[0,\infty))$ be as in   Lemma \ref{lem:abc} with constants $a,b\in(0,\infty)$.  For $M\in\N$ and $y_1,\dots,y_M\in X$, define
\begin{align*}
&u(\mu):=\Phi(\mu(\gamma\circ g )),\ \ \  \mu\in\scr P_p,\\
&g(x):=\min \big\{ \|x-y_j\|_X\,:\,1\le j\le M\big\},\quad x\in X.
\end{align*}
Then $u\in \D(\EE)$ with  $\EE(u,u)\leq C(ab)^2$.
\end{prp}
\begin{proof} 
	By  Lemma \ref{lem:abc}, we need to approximate $g$ by functions of the type $f(x_1',\cdots, x_m')$ with $m\in\N,
x_1',\cdots, x_m'\in X^*$ and Lipschitz continuous $f$ on $\R^m$. This is achieved in the second step of this proof (in (b) below)
using the following preliminary remark (a): 

(a) For $m\in\N, r\in\R^m$ and subsets $I_1,\dots, I_k$ of $\{1,\dots, m\},$ the function
\begin{equation*}
f(z):=\min \Big\{ \max_{i\in I_j}(z_i+r_i)\,: \,j=1,\dots,k\Big\},\quad z\in\R^m
\end{equation*}
is Lipschitz continuous  with
$$\sum_{i=1}^m|\partial_if(z)|=1\quad\text{for }\d z\text{-a.e.~}z\in\R^m.$$
By Lemma \ref{lem:abc}, the associated function $u: \scr P_p\to\R$ belongs to  $\D(\EE)$ with   $\EE(u,u)\leq C(ab)^2.$

(b)	Let $\{x_k:\, k\in\N\}$ be a dense subset of $X$. For each $k\in\N$ we choose $x_k'\in X^*$ with ${\|x_k'\|}_{X^*}=1$ and $\phantom{}_{X^*}{\langle}x_k',x_k\rangle_X={\|x_k\|}_X$.
Moreover, let $M\in\N$, $y_1,\dots,y_M\in X$ and
\begin{equation*} g_l(x):=\min_{j=1,\dots,M}\Big(\max_{k=1,\dots,l}\,\phantom{}_{X^*}{\langle}x_k',x-y_j\rangle_X\Big),\quad x\in X,\,l\in\N.
\end{equation*}
As shown in step (a) in the present proof,  the function
\begin{equation*}
u_l(\mu):=\Phi(\mu(\gamma\circ g_l))\quad\text{for }\ \mu\in\scr P_p
\end{equation*} is in $\D(\EE)$ with   $\EE(u_l,u_l)\leq C(ab)^2.$
By construction of $x_k'$, $k\in\N$, it holds $\sup_{k\in\N}\phantom{}_{X^*}{\langle}x_k',x\rangle_X=\|x\|_X$ for $x\in X$.
Consequently,
\begin{equation*}
\lim_{l\to\infty} g_l(x)=\min\big\{ \|x-y_i\|_X\,:\,j\in\{1,\dots,M\}\big\}=:g(x)\quad\text{for }x\in X.
\end{equation*}
Since $\sup_{l\in\N}|g_l(x)|\leq \max_{j=1,\dots,M}\|x-y_j\|_X,$ we find  constants $c_1,c_2 \in (0,\infty)$ such that
\begin{equation*}
\big|\gamma\circ  g(x)\big|\leq c_1\big(1+|g(x)|^p\big) \leq c_2\big(1+ \|x\|_X)^p\big),\quad x\in X.
\end{equation*}
So,  applying Lebesgue's dominated convergence twice as in the proof of  Lemma \ref{lem:abc}, we obtain
$$\lim_{l\to\infty}\|u_l-u\|_{L^2(\scr P_p,\LL)}= 0.$$  Now, the claim of this proposition is a consequence of  \cite[Lemma I.2.12]{MR92}
and the estimate $\EE(u_l,u_l)\le C(ab)^2$.
\end{proof}

%We are now ready to prove the criterion.

\begin{proof}[Proof of Theorem \ref{QRT}]
Lemma \ref{lem:onBorel}(i) together with   $(C_1)$  yields the existence of a countable set $\{f_i\}_{i\in\N}$ of quasi-continuous functions in $\D(\EE)$   which
separate points in $\scr P_p$.
So,  it suffices  to find a $\EE$-nest of compact sets.

First, we recall a characterization of precompact sets in $\scr P_p$  as stated in \cite[Proposition 2.2.3]{PZ20}. The closure w.r.t.~$\W_p$ of a set $\scr A\subseteq \scr P_p$ is compact if and only
\begin{itemize}
\item[(1)] (Uniform Integrability) $\displaystyle\lim_{R\to\infty}\displaystyle\sup_{\mu\in \scr A}\mu\big(\|\,\cdot\,\|_X^p\eins_{[R,\infty)}(\|\,\cdot\,\|_X\big)\big) =0,$
\item[(2)] (Tightness) for every $\varepsilon>0$ there is a compact set $Y\subseteq X$ such that $\displaystyle\sup_{\mu\in \scr A}\mu(Y^c)\le  \varepsilon$.
\end{itemize}
With this characterization, we only need to find   an $\EE$-nests of closed sets $\{K_n^{(i)}\}_{n\in\N}$ for $i=1,2,$ such that (1) holds for $\scr A=K_n^{(1)}$ and (2) holds for $\scr A=K_n^{(2)}$.
If this is achieved,  $\{K_n:=K_n^{(1)}\cap K_n^{(2)}\}_{n\in\N}$ is a $\EE$-nest of compact sets, and hence the proof is finished.
Indeed, as $\text{Cap}_1$ is a Choquet capacity on $\scr P_p$, it holds
\begin{align*}
\text{Cap}_1(X\setminus K_n)&=
\text{Cap}_1\big(\big(X\setminus K_n^{(1)}\big)\cup\big( X\setminus K_n^{(2)}\big) \big)\\
&\leq \text{Cap}_1\big(X\setminus K_n^{(1)}\big)+\text{Cap}_1\big(X\setminus K_n^{(2)}\big)
\overset{n\to\infty}{\longrightarrow} 0.
\end{align*}

The remainder of this proof is divided into two parts. In (a) we construct the family $\{K_n^{(1)}\}_{n\in\N}$, and in (b) 
we construct the family $\{K_n^{(2)}\}_{n\in\N}$.

(a)  For $k\in\N$ we let
\begin{equation}\label{eqn:ukD}
u_k(\mu):=\chi_1\big(\mu\big(\gamma_k( \|\,\cdot\,\|_X)\big)\big),\quad  \mu\in\scr P_p,
\end{equation}
where $\chi_1$ is the function from \eqref{eqn:chi} with $l=1$ and
$$\gg_k(s):= \big(1+[(s-k)^+]^2\big)^{\ff p 2}-1,\ \ s\in\R.$$
It is easy to find constants $a,b\in (0,\infty)$ independent of $k$ such that conditions in
Lemma \ref{lem:abc} holds for $(\Phi,\gg)=(\chi_1, \gg_k), k\in \N.$
So,  Proposition \ref{prp:abc} implies
\begin{equation*}
\sup_{k\in\N}\EE(u_k,u_k)<\infty,
\end{equation*}
and  by Lebesgue's dominated convergence theorem twice as in the proof of Lemma \ref{lem:abc}, we have
$  u_k(\mu)\to 0$ as $k\to\infty$ for each $\mu\in \scr P_p,$ and
$$\lim_{k\to\infty} \LL(|u_k|^2)=0.$$
By \cite[Lemma I.2.12]{MR92},  there exists a subsequence $\{u_{k_l}\}_{l\in\N}$ such that
\begin{equation}
v_m:= \ff 1 m \sum_{k=1}^m u_{l_k}\overset{m\to\infty}{\longrightarrow}0
\end{equation}
strongly in the Hilbert space $(\D(\EE),\EE_1)$.
Then, by \cite[Proposition III.3.5]{MR92}, there exists an $\EE$-nest of closed sets   $\{K^{(1)}_n\}_{n\in\N}$ and a subsequence $\{v_{m_k}\}_{k\in\N}$ such that
$$\lim_{k\to\infty} \sup_{\mu\in K_n^{(1)}} v_{m_k}(\mu) =0\quad\text{for every }n\in\N.$$
Since $\{u_k\}_{k\in\N}$ is a decreasing sequence, it holds  $v_{m}(\mu)\ge u_{l_{m}}(\mu)$ for every $m\in\N$ and $\mu\in\scr P_p$.
In particular,  $v_{m_k}(\mu)\ge u_{l_{m_k}}(\mu)$ for $\mu\in\scr P_p$ and $k\in\N$.
Now, it follows
\begin{equation*}
\lim_{k\to\infty} \sup_{\mu\in K_n^{(1)}} u_{l_{m_k}}(\mu) =0\quad\text{for every }n\in\N,
\end{equation*}
which in turn implies
\begin{equation}\label{*3}
\lim_{k\to\infty} \sup_{\mu\in K_n^{(1)}} \mu\big(\gamma_{l_{m_k}}(\|\,\cdot\,\|_X)\big)
=0\quad\text{for every }n\in\N.
\end{equation}
Next, for $k\in\N$ and $s\in[2k+2,\infty)$, 
\begin{equation*}
	\gg_k(s)\geq \Big[1+\big(\tfrac{s}{2}+1\big)^2\Big]^{\frac p 2 }-1\geq {\big(\tfrac{s}{2}+1\big)}^p-1\geq \frac{s^p}{2^p}
\end{equation*} 
by the monotonicity of $\gg_k$.
So, choosing $R_k:=2k+2$, $k\in\N$, results in
$$2^p\mu\big(\gamma_k(\|\,\cdot\,\|_X)\big)\ge\mu\big(\|\,\cdot\,\|_X^p\, \eins_{[R_k,\infty)}(\|\,\cdot\,\|_X)\big),\quad\mu\in\scr P_p.$$
This together with \eqref{*3} yields that for every $n\in\N$,
\begin{align*}
\liminf_{k\to\infty}\,\sup_{\mu\in K_n^{(1)}}\mu\big(\|\,\cdot\,\|_X^p\eins_{[R_k,\infty)}(\|\,\cdot\,\|_X\big)\big)
\leq 2^p\liminf_{k\to\infty}\,\sup_{\mu\in K_n^{(1)}} \mu\big(\gamma_{k}(\|\,\cdot\,\|_X)\big)=0.
\end{align*}
Thus,  (1) holds for $\scr A= K_n^{(1)}$ as desired.

(b) We construct the family $\{K_n^{(2)}\}_{n\in\N}.$   Let $\{y_i\}_{i\in\N}$ be a countable dense subset of $X$, and for each   $k\in\N$, let
\begin{equation*}
u_k(\mu):=\mu\Big(\chi_1\big(\min\big\{ \|\,\cdot\,-y_i\|_X\,: \,i=1,\dots,k\big\} \big) \Big),\quad  \mu\in\scr P_p.
\end{equation*}
Again, $\chi_1$ is the function from \eqref{eqn:chi} with $l=1$. Noting that
$\chi_1\big(\min\big\{ \|\,\cdot\,-y_i\|_X\,: \,i=1,\dots,k\big\} \big) $ decreases to $0$ as $k\to\infty$,
by Proposition \ref{prp:abc} and
the dominated convergence theorem, we have
\begin{equation*}
\sup_{k\in\N}\EE(u_k,u_k)<\infty,\ \ \lim_{k\to\infty} \LL(|u_k|^2)=0.
\end{equation*}
So, as shown above    \cite[Lemma I.2.12]{MR92} and \cite[Proposition III.3.5]{MR92} imply the existence of a subsequence  of $\{u_k\}_{k\in\N}$ which converges to zero quasi-uniformly. The arguments, which include the strong convergence w.r.t.~$\EE_1^{1/2}$ of the C\'esaro means for a suitable subsequence together with the fact that $\{u_k\}_{k\in\N}$ is a decreasing in $k$, are completely analogous to step (a). Hence, there exists a nest $\{K_n^{(2)}\}_{n\in\N}$ such that for each $n\in\N$ the following property holds:
\begin{equation}\label{eqn:km}
\text{For }\varepsilon>0 \text{ there exists }\{k_m\}_{m\in\N}\subseteq \N\text{ with }\sup_{\mu\in K_n^{(2)}}u_{k_m}(\mu)\leq \frac{\varepsilon}{m2^m},\ \ m\in\N.
\end{equation}
It remains to show that $K_n^{(2)}$ satisfies property (2) for fixed $n\in\N$. Let $m\in\N$, $\varepsilon>0$ be arbitrary and $k_m$ be chosen according to \eqref{eqn:km}. Since for $r\leq 1$ and $s\in[0,\infty)$ it holds
$$\chi_1(s)\geq r\quad\text{if and only if}\quad s\geq r,$$
we estimate
\begin{align}\label{eqn:e2m}
&\mu\Big(\Big\{x\in X\,: \,\min\big(\big\{ \|\,x-y_i\|_X\,:\,i=1,\dots,k_m\big\} \geq \frac{1}{m} \Big\}\Big)\\\nonumber
&=\mu\Big(\Big\{x\in X\,: \,\chi_1\big(\min \big\{ \|\,x-y_i\|_X\,:\,i=1,\dots,k_m\big\} \big)\geq \frac{1}{m} \Big\}\Big)
\\&\leq m \,u_{k_m}(\mu)\leq \frac{\varepsilon}{2^m},\ \ m\in\N,\ \mu\in K_n^{(2)}.\nonumber
\end{align}
Now we define $Y:=\bigcap_{m\in\N}Y_m$ for
$$Y_m:=\Big\{x\in X\,: \,\min \big\{ \|\,x-y_i\|_X\,:\,i=1,\dots,k_m\big\} < \frac{1}{m} \Big\}.$$  Obviously, $Y$ is a totally bounded set in $X$ and hence the closure $\overline Y$ a compact set. This proves the tightness of $K_n^{(2)}$, because for any $\mu\in K_n^{(2)}$ using \eqref{eqn:e2m} it holds
\begin{align*}
\mu(X\setminus \overline Y)\leq \mu\Big(\bigcup_{m\in\N}X\setminus Y_m\Big)\leq
\sum_{m=1}^\infty\mu(X\setminus Y_m)\leq \sum_{m=1}^{\infty}\frac{\varepsilon}{2^m}=\varepsilon
\end{align*}
and the choice of $\varepsilon$ above is arbitrary.
\end{proof}

\section{Quasi-regular image Dirichlet forms on \texorpdfstring{$\scr P_p$}{Pp}}\label{sec:gradForm}

\noindent In this section we prove the quasi-regularity for image Dirichlet forms on $\scr P_p$ under the map
\begin{equation}\label{eqn:PsiDef}
\Psi: T_{\mu_0,p}\ni \phi\mapsto \mu_0\circ \phi^{-1} \in \scr P_p
\end{equation}
from the tangent space $T_{\mu_0,p}$ for a fixed element   $\mu_0\in \scr P_p$.
To shorten notation, we set
$$T_0:=T_{\mu_0,p}=L^p(X\to X,\mu_0),\ \ T_0^*:= T_{\mu_0,p}^*=L^{p^*}(X\to X^*, \mu_0).$$
The map $\Psi$  is Lipschitz continuous, since
\begin{equation*}
\W_p(\Psi(\phi_1),\Psi(\phi_2))^p 
\leq \int_{X}\|\phi_1(x)-\phi_2(x)\|_X^p\,\d\mu_0(x)
= \|\phi_1-\phi_2\|_{T_{0}}^p,\qquad \phi_1,\phi_2\in T_0.
\end{equation*}
In the case where $X$ is a separable Hilbert space, $p>1$ and  $\mu_0\in\scr P_p$ is  absolutely continuous w.r.t.~a  non-degenerate Gaussian measure,
the theory of optimal transport provides that
$\Psi$ is surjective and that for any $\mu\in\scr P_p$ there exists a unique  $\phi_\mu\in T_{0}$ such that
\begin{equation*}
\mu_0\circ\phi_\mu^{-1}=\mu\quad\text{and}\quad \W_p(\mu_0,\mu)=\|\textnormal{id}-\phi_\mu\|_{T_{0}}
\end{equation*} 
(see  \cite[Theorem~6.2.10]{AGS05}).
Here, $\phi_\mu$ is called the optimal map from $\mu_0$ to $\mu.$
In particular, this holds if $X=\R^d$ and $\mu_0\in \scr P_p$ is absolutely continuous w.r.t.~the Lebesgue measure.

Let $\LL_0$ be a probability measure on $T_{0}$.  Then, the image measure $$\LL:=\LL_0\circ\Psi^{-1}$$ is a probability measure on $\scr P_p$. The map
$$L^2(\scr P_p,\LL)\ni u\mapsto u\circ\Psi\in L^2(T_0,\LL_0)$$ is isometric, i.e.
\beq\label{WW} \LL(uv)=\LL_0\big((u\circ\Psi)(v\circ\Psi)\big),\ \ \ u,v\in L^2(\scr P_p,\LL).\end{equation}
We would like to infer that, if $\Psi$ is surjective, then choosing $\LL_0$ such that
$$\LL_0(U)>0\quad\text{for any non-empty open set}\quad U\subseteq T_0$$
results in $\LL$ bearing the analogous property, i.e.~it takes a strictly positive value on each non-empty open set in $\scr P_p$. We refer to this property of $\LL$ (resp.~$\LL_0$) as full 
topological support. An example for such a measure is found in Section \ref{sec:exaFTS}.
In general, since the topology of $\scr P_p$ is second countable and in particular strongly Lindel\"of, the support of a measure on the Borel $\sigma$-algebra is well-defined. For a measurable function $u:\scr P_p\to\R$ we set $\supp[u]:=\supp[|u|\LL]$.

In the following, a class of quasi-regular Dirichlet forms   in $L^2(\scr P_p, \LL)$ are constructed by using the image of Dirichlet forms  in $L^2(T_0,\LL)$. This provides  quasi-regular local Dirichlet forms associated with diffusion processes on $\scr P_p$. 

\subsection{Main result}
\noindent As declared in Section \ref{sec:QR}, all Dirichlet forms we consider are assumed to be symmetric. Let
$(\widetilde\EE,\D(\widetilde\EE))$ be a symmetric Dirichlet form in $L^2(T_{0},\LL_0)$, where $\D(\widetilde\EE)$ contains the following class $C_b^1(T_{0})$ of functions on $T_0$.

\beg{defn}\label{def:Cb1T}   $C_b^1(T_{0})$   consists of all bounded Fr\'echet differentiable functions $f$ on $T_0$
with
$$T_0\ni\phi\mapsto \nabla f(\phi)\in L^{1}(X\to X^*,\mu_0) $$ continuous,   and
\begin{equation*}
\sup_{\phi\in T_0}\|\nabla f(\phi)\|_{L^\infty(X\to X^*,\mu_0)}<\infty.
\end{equation*}\end{defn}

Note that by the dominated convergence theorem,  if $f\in C_b^1(T_{0})$, then
$$T_0\ni\phi\mapsto \nabla f(\phi)\in L^{q}(X\to X^*,\mu_0) $$ is continuous for all $q\in [1,\infty).$
The main result of this section is the following.

\begin{thm}\label{TN}
Let $\LL_0$ be Borel probability measure on $T_0$, and let $(\tt\EE,\D(\tt \EE))$ be a symmetric Dirichlet form
in $L^2(T_0,\LL_0)$ such that $ C_b^1(T_{0})\subseteq \D(\tt \EE).$
We have the following assertions.
\beg{enumerate} \item[$(1)$] It holds
\beq\label{TU}   C_b^1(\scr P_p)\circ\Psi:=\big\{u\circ\Psi\,:\, u\in C_b^1(\scr P_p)\big\} \subseteq C_b^1(T_{0}),\end{equation}
and the bilinear form
\begin{equation*}\EE(u,v):= \tt\EE(u\circ\Psi, v\circ \Psi),\quad u,v\in C_b^1(\scr P_p)\end{equation*}
is closable in $L^2(\scr P_p,\LL)$, where $\LL:=\LL_0\circ\Psi^{-1}$.  Its  closure
$(\EE,\D(\EE))$ is a  Dirichlet form satisfying
\beq\label{EE'} \beg{split}&\D(\EE)\circ\Psi:=\big\{u\circ\Psi:\ u\in \D(\EE)\big\}\subseteq \D(\tt\EE),\\
&\EE(u,v)=\tt\EE(u\circ\Psi, v\circ\Psi),\quad u,v\in \D(\EE).\end{split}\end{equation}
\item[$(2)$]   If the generator $(\tt L,\D(\tt L))$ of $(\tt\EE,\D(\tt\EE))$ has a purely discrete spectrum, let 
${\{\lambda_n\}}_{n\in\N}$ denote all eigenvalues of $-\tt L$ listed in the increasing order with multiplicities.
In this case, the generator $(L,\D(L))$ of $(\EE,\D(\EE))$ also has a purely discrete spectrum, and the eigenvalues ${\{\lambda_n\}}_{n\in\N}$ of $-L$ listed in the same order satisfy $\lambda_n\ge \si_n$.
If additionally $\sum_{n\in\N} \e^{-\si_n t}<\infty$ for $t>0$, then the  semigroup $T_t:=\e^{t L}$ has a heat kernel $p_t(\mu,\nu)$ with respect to $\LL$ satisfying
\begin{equation*}
\int_{\scr P_p\times\scr P_p}p_t(\mu,\nu)^2\d\LL(\mu)\d\LL(\nu)\le \sum_{n\in\N}\e^{-2\ll_n t}<\infty,\quad t>0.
\end{equation*}
\item[$(3)$] If there exists a constant $C>0$ such that
\beq\label{QQ} \tt\EE(f,f)\le C \sup_{\phi\in T_0} \|\nn f(\phi)\|_{T_0^*}^2,\qquad f\in C_b^1(T_{0}),\end{equation}
then $(\EE,\D(\EE))$ is quasi-regular.\end{enumerate}
\end{thm}

To prove \eqref{TU}, we need the following chain rule.

\begin{lem}\label{lem:chain}
If $u\in C_b^1(\scr P_p)$, then
$u\circ \Psi\in C_b^1(T_{0})$ with
\begin{equation*}
\nabla (u\circ\Psi)(\phi)=Du(\Psi(\phi))\circ \phi, \quad\phi\in T_{0}.
\end{equation*}
In particular,
\begin{equation*}
\big\|\nabla (u\circ\Psi)(\phi)\big\|_{T_0^*}=\big\|Du(\Psi(\phi))\big\|_{T_{\Psi(\phi),p}^*},\quad \phi\in T_0.
\end{equation*}
%	\item[(ii)] If $f\in \scr FC_b^1(\scr P)$, then  $\nabla(f\circ \Psi)$ is a continuous map from $T_0$ into $L^\infty(X,\mu_0,X^*)$. In particular, $\nabla(f\circ \Psi)$ is a continuous map from $T_0$ into $T_0^*$ (in any situation $1\leq p <\infty$).
\end{lem}

\begin{proof}
Let $\phi,\xi\in T_0$. Applying   \cite[Theorem 2.1]{BRW}  on the probability space $(X,\scr B(X),\mu_0)$ implies
\beg{align*} & (u\circ\Psi)(\phi+ \xi)-(u\circ\Psi)(\phi)= \int_0^1\ff{\d}{\d\vv} (u\circ\Psi)(\phi+ \vv\xi)\,\d\vv\\
&= \int_0^1 \phantom{}_{T_0^*}{\big\langle} Du(\Psi(\phi+\vv\xi))\circ(\phi+\vv\xi),
\xi{\big\rangle}_{T_0}.\end{align*}
Combining this with the boundedness and continuity of $Df$ on $\scr P_p\times X$, we may apply the dominated convergence theorem to deduce
\begin{equation}\label{eqn:fd2}
\lim_{\|\xi\|_{T_0}\downarrow   0} \bigg|\frac{(u\circ\Psi)(\phi+ \xi)-(u\circ\Psi)(\phi)- \phantom{}_{T_0^*}{\big\langle} Du(\Psi(\phi))\circ\phi,
\xi{\big\rangle}_{T_0}} {\|\xi\|_{T_0}} \bigg| =0,
\end{equation}
hence  $u\circ\Psi$ is Fr\'echet differentiable on $T_{0}$ with derivative
\begin{equation}\label{eqn:chainR}
\nabla (u\circ \Psi)(\phi)=Du(\Psi(\phi))\circ\phi\in T_0^*
\end{equation} satisfying
$$\|\nn  (u\circ \Psi)(\phi)\|_{L^\infty(X\to X^*,\mu_0)}\le \|Du\|_\infty<\infty.$$

Finally, we fix $\{\phi_n\}_{n\in\N}\subseteq T_0$ and $\phi\in T_0$ such that 
\begin{equation*}
\lim_{n\to\infty}\|\phi_n-\phi\|_{T_0}=\lim_{n\to\infty}\bigg(\int_{X}\|\phi_n-\phi\|_{X}^p\d\mu_0\bigg)^{\ff 1 p}= 0.
\end{equation*}
By the continuity of $\Psi: T_0\to \scr P_p$, and  the boundedness and continuity of   $Du:\scr P_p\times X\to X^*$, we may apply  the dominated convergence theorem to derive
$$\lim_{n\to\infty} \int_X\|Du(\Psi(\phi_n))(\phi_n)- Du(\Psi(\phi))(\phi)\|_{X^*} \d\mu_0=0,$$ which
together with \eqref{eqn:chainR} yields the continuity of $T_0\ni\phi\mapsto \nn u(\phi)\in L^1(X\to X^*,\mu_0)$.

\end{proof}

\begin{proof}[Proof of Theorem \ref{TN}] (1) The  inclusion  \eqref{TU} is ensured by \eqref{WW} and Lemma \ref{lem:chain}. Next, $\scr FC_b^1(\scr P_p)\subseteq C_b^1(\scr P_p)$ and Lemma \ref{L2} 
imply that	 $C_b^1(\scr P_p)$ is dense in $L^2(\scr P_p,\LL)$.
Hence, due to $C_b^1(T_{0})\subseteq\D(\tt\EE)$ and \eqref{TU},  also
$$\Psi^*\D(\widetilde\EE):=\big\{u\,:\,u\circ\Psi\in\D(\widetilde\EE)\big\}\supseteq C_b^1(\scr P_p)$$
is a dense subset of $L^2(\scr P_p,\LL).$ So, by
\cite[Chapt.~V]{bou91},
\begin{equation*}
\Psi^*\widetilde\EE(u,v):=\widetilde\EE(u\circ\Psi,v\circ\Psi),\ \  u,v \in\Psi^*\D(\widetilde\EE),
\end{equation*}
is  a Dirichlet form in $L^2(\scr P_p,\LL)$. Since $\EE(u,v)=\widetilde\EE(u\circ\Psi,v\circ\Psi)$ for $u,v\in C_b^1(\scr P_p)$, the bilinear form
$(\EE, C_b^1(\scr P_p))$ is densely defined and closable  in $L^2(\scr P_p,\LL)$. Moreover, 
its closure $(\EE,\D(\EE))$ is a Dirichlet form, as follows from the Markovian property of $\Psi^*\tilde\EE$ together with $\tau\circ u\in C_b^1(\scr P_p)$ for all $\tau\in C_b^1(\R)$ and $u\in C_b^1(\scr P_p)$.

(2)    We assume that $\widetilde L$ has a purely discrete spectrum. We use the same argument as in \cite[Thm.~3.2(2)]{RW22} to show that $L$ has a purely discrete spectrum.
The latter is equivalent to the statement that
\begin{equation}\label{eq:compactball}
\big\{  u\in  \D(\EE):\EE(u,u)+\Lambda(u^2)\leq 1\big\}\text{ is compact w.r.t.~the trace topology of }L^2(\scr P_p,\Lambda).
\end{equation}
%The spectrum of $\widetilde L$ consists of isolated points with finite multiplicity. Hence, the set
By assumption
\begin{equation*}
	K:=\big\{ f\in\D(\widetilde\EE): \widetilde \EE(f,f)+\Lambda_0(f^2)\leq 1\}
\end{equation*}
is compact w.r.t.~the trace topology of $L^2(T_0,\Lambda_0)$. We have
\begin{equation*}
	\big\{ u\circ\Psi: u\in  \D(\EE),\, \EE(u,u)+\Lambda (u^2)\leq 1\big\}\subseteq K
\end{equation*}
because of \eqref{EE'}  and furthermore $u\mapsto u\circ\Psi$ is an isometry between $L^2(\scr P_p,\Lambda)$ and $L^2(T_0,\Lambda_0)$. This yields \eqref{eq:compactball}.

Next, by \eqref{EE'} and the Courant-Fischer min-max principle, for any $n\in\N$ we have
\beg{align*} \ll_n&= \inf_{\C: n\text{-dim.\  subspace\ of \ }\D(\EE)}\   \sup_{0\ne u \in \C} \ff{\EE(u,u)}{\LL(u^2)}  \\
&=\inf_{\C: n\text{-dim.\  subspace\ of \ }\D( \EE)} \  \sup_{0\ne u \in \C} \ff{\widetilde\EE(u\circ \Psi,u\circ\Psi)}{\LL_0((u\circ\Psi)^2)}\\
&\ge \inf_{\tt\C: n\text{-dim.\  subspace\ of \ }\D(\widetilde \EE)} \  \sup_{0\ne \tt u \in \tt\C} \ff{\widetilde \EE(\tt u,\tt u)}{\LL_0(\tt u^2)}=\sigma_n.   \end{align*}
So, if
$\sum_{n\in\N} \e^{-2\sigma_n t} <\infty$ for $t>0$,
by the spectral representation (see for instance \cite{Davies}),  $T_t$ has a  heat kernel $p_t$ w.r.t.~$\LL$ given by
\begin{equation*}p_t(\mu,\nu):=\sum_{n\in\N} \e^{-\ll_n t} u_n(\mu)u_n(\nu), \quad \mu,\nu\in \scr P_p,\end{equation*}
where ${\{u_n\}}_{n\in\N}$ is an orthonormal eigenbasis of $-L$ corresponding to ${\{\lambda_n\}}_{n\in\N}$.
Hence,
\begin{equation*}\int_{\scr P_p\times\scr P_p} p_t(\mu,\nu)^2 \,\d\LL(\mu)\,\d\LL(\nu)= \sum_{n\in\N} \e^{-2\ll_n t}
<\infty.\end{equation*}

(3) Since $C_b^1(\scr P_p)$ is a dense subspace of $\D(\EE)$, the proof of (1) implies Condition  $(C_2)$ in Theorem \ref{QRT}. Moreover, Lemma \ref{lem:chain}, \eqref{QQ} and \eqref{WW} imply $(C_1)$. So,  the quasi-regularity of $(\EE,\D(\EE))$ follows from   Theorem \ref{QRT}.\end{proof}

\subsection{Local  Dirichlet forms and diffusion processes}\label{diffusion}

\noindent In classic theory, the Dirichlet form for a symmetric diffusion process on $\R^d$ is of gradient type
$$\EE(f,g)=\int_{\R^d}\<Q\nn f,\nn g\>_{\R^d}\,\d\LL$$
for a nice probability measure $\LL$ on $\R^d$ and a diffusion coefficient $Q=(q_{ij})_{1\le i,j\le d}$.

In the following, we consider the case $X=H$ for a separable Hilbert space $H$ and develop an analogous concept for the state space $\scr P_p$, $p\in[1,2]$. 
First, we fix a symmetric bounded linear operator $Q$ on $L^2(T_{0}\to T_{\mu_0,2},\LL_0)$ such that there is a decomposition $Q=\{Q_\phi\}_{\phi\in T_0}$ with $Q_\phi:T_{\mu_0,2}\to T_{\mu_0,2}$ and
\begin{equation}\label{eq:dec}
	QV(\phi):=Q_\phi (V(\phi))\quad\text{for}\quad V\in L^2(T_{0}\to T_{\mu_0,2},\LL_0).
\end{equation}
By Definition \ref{def:Cb1T} we have $Q\nabla f\in L^2(T_{0}\to T_{\mu_0,2},\LL_0)$ for $f\in C_b^1(T_0)$. 
We assume that  
\begin{equation}\label{eq:preim}
	\tt\EE(f,g)=\int_{T_0}\langle Q_\phi(\nabla f(\phi)),\nabla g(\phi)\rangle_{T_{\mu_0,2}}\d\LL_0(\phi),\qquad f,g\in C_b^1(T_0),
\end{equation}
is a well-defined and closable bilinear form in $L^2(T_0,\LL_0)$.
Let
\begin{equation}\label{image}
\EE(u,v):=\tt\EE(u\circ\Psi,v\circ\Psi),\quad u,v\in C_b^1(\scr P_p).
\end{equation}

\begin{rem}\label{rem:image}
	The form in \eqref{image} is closable in $L^2(\scr P_p,\LL)$ and its closure $(\EE,\D(\EE))$ is quasi-regular by Theorem \ref{TN}.
	Moreover, we note that
	\begin{equation*}
L^{p^*}(H\to H,\mu)\subseteq L^2(H\to H,\mu)\subseteq L^p(H\to H,\mu),
	\end{equation*}
	i.e.~$T_{\mu,p}^*\subseteq T_{\mu,2}\subseteq T_{\mu,p}$ for $\mu\in\scr P_p$, since $p\in[1,2]$.
\end{rem}

\beg{thm}\label{T2} If \eqref{eq:preim} is well-defined and closable in $L^2(T_0,\LL_0)$, then
$(\EE,\D(\EE))$ as defined in \eqref{image} and Remark \ref{rem:image} is a local, quasi-regular Dirichlet form in $L^2(\scr P_p,\LL)$. Moreover, there are bounded symmetric operators $B_\mu:T_{\mu,2}\to T_{\mu,2}$, $\mu\in\scr P_p$,
such that the representation
\begin{equation*}
	\EE(u,v)=\int_{\scr P_p}\Gamma (u,v)(\mu)\d\LL(\mu)\quad\text{for }u,v\in C_b^1(\scr P_p)
\end{equation*}
holds with square-field operator 
\begin{equation}\label{eq:DefG}
	\Gamma(u,v)(\mu)=\langle D u(\mu),B_\mu( D v(\mu))\rangle_{T_{\mu,2}}.
\end{equation}
\end{thm}
\beg{proof} 
In part (a) we show the claimed representation for $\EE$. In (b) the local property of $\EE$ is addressed.

(a) The direct integral ${\textstyle\int}^\oplus_{\scr P_p} T_{\mu,2}\d\LL(\mu)$ contains all $\LL$-classes of measurable vector fields $V(\mu)\in T_{\mu,2}$, $\mu\in\scr P_p$, such that
$\int_{\scr P_p} \|V(\mu)\|_{T_{\mu,2}}^2\d\LL(\mu)<\infty$ (see \cite{Dixmier, Taka}). The inner product is given by
	\sloppy $\int_{\scr P_p} {\langle V(\mu),W(\mu)\rangle}_{T_{\mu,2}}^2\d\LL(\mu)$ for $V,W\in {\textstyle\int}^\oplus_{\scr P_p} T_{\mu,2}\d\LL(\mu)$.
The pull-back operator 
\begin{equation*}
	PV(\phi):=V(\Psi(\phi))\circ\phi,\quad \phi\in T_0,
\end{equation*}
linearly and isometrically maps an element $V\in {\textstyle\int}^\oplus_{\scr P_p} T_{\mu,2}\d\LL(\mu)$ to $L^2(T_0\to T_{\mu_0,2},\LL_0)$.
It has an adjoint operator $P^*:L^2(T_0\to  T_{\mu_0,2},\LL_0)\to{\textstyle\int}^\oplus_{\scr P_p} T_{\mu,2}\d\LL(\mu)$
and we can define $B:=P^*QP$, a symmetric, bounded linear operator on ${\textstyle\int}^\oplus_{\scr P_p} T_{\mu,2}\d\LL(\mu)$.
For $V,W\in{\textstyle\int}^\oplus_{\scr P_p} T_{\mu,2}\d\LL(\mu)$ and $\omega\in C_b(\scr P_p)$, by definition
\begin{align*}
	&\int_{\scr P_p}\langle B(\omega V)(\mu),W(\mu)\rangle_{T_{\mu,2}}\d\LL(\mu)=\int_{T_0}\langle QP(\omega V)(\phi),PW(\phi)\rangle_{T_{\mu_0,2}}\d\LL_0(\phi)\\
	&=\int_{T_0}\big\langle Q_\phi \big[\omega(\mu_0\circ\phi^{-1}) V(\mu_0\circ\phi^{-1})\circ\phi\big],W(\mu_0\circ\phi^{-1})\circ\phi \big\rangle_{T_{\mu_0,2}}\d\LL_0(\phi)\\
	&=\int_{T_0} \omega(\mu_0\circ\phi^{-1})\big\langle Q_\phi \big[ V(\mu_0\circ\phi^{-1})\circ\phi\big],W(\mu_0\circ\phi^{-1})\circ\phi \big\rangle_{T_{\mu_0,2}}\d\LL_0(\phi)\\
	&=\int_{T_0}\langle Q(PV)(\phi),P(\omega W)(\phi)\rangle_{T_{\mu_0,2}}\d\LL_0(\phi)=
	\int_{\scr P_p}\langle \omega(\mu) B V(\mu),W(\mu)\rangle_{T_{\mu,2}}\d\LL(\mu),
\end{align*}
i.e.~$B$ commutes with the multiplication operator $ V\mapsto \omega V$.
Hence, there exists a decomposition $B=\{B_\mu\}_{\mu\in\scr P_2}$, where $B_\mu: T_{\mu,2}\to T_{\mu,2}$ is a symmetric bounded operator (see \cite[Sect.~II.2.5]{Dixmier}) and
\begin{equation*} 
	BV(\mu)=B_\mu (V(\mu)),\qquad  \mu\in\scr P_p,\,V\in{\textstyle\int}^\oplus_{\scr P_p} T_{\mu,2}\d\LL(\mu).
\end{equation*}
Using Lemma \ref{lem:chain} together with the above equation for the choices $\omega=1$, $V=Du$, $W=Dv$ for given $u,v\in C_b^1(\scr P_2)$, we obtain
\begin{multline*}
	\EE(u,v)=\tt\EE(u\circ\Psi,v\circ\Psi)=\int_{T_{0}}\langle QPV,PW\rangle_{T_{\mu_0,2}}\d\LL_0
	\\=\int_{\scr P_p}\langle B(Du)(\mu),Dv(\mu)\rangle_{T_{\mu,2}}\d\LL(\mu)=\int_{T_0}\langle B_\mu (Du(\mu)),Dv(\mu)\rangle_{T_{\mu,2}}\d\LL(\mu).
\end{multline*}
The fact that $\Gamma$ in \eqref{eq:DefG} is indeed the square-field operator of $\EE$ in the sense \cite[Sect.~I.4]{bou91} follows easily through
\begin{multline*}
	\Gamma(u,v)(\Psi(\phi))=\langle Q_\phi Du(\Psi(\phi))\circ\phi, Dv(\Psi(\phi))\circ\phi \rangle_{T_{\mu_0,2}}
	\\=\langle Q_\phi \nabla(u\circ\Psi)(\phi), \nabla(u\circ\Psi)(\phi) \rangle_{T_{\mu_0,2}}
	=:\tilde\Gamma(u\circ\Psi,v\circ\Psi)(\phi)
\end{multline*}
for $\phi\in T_0$, $u,v\in C_b^1(\scr P_2)$ and the corresponding property for $\tilde\Gamma$ in relation to $\tt\EE$.

(b) To prove the locality of  $(\EE, \D(\EE))$,   it suffices to show that
\beq\label{FA}   {\rm supp} [u\circ \Psi]  \subseteq   \Psi^{-1} \big(
{\rm supp}[u]\big)
\quad \text{for }u\in L^2(\scr P_p,\LL).\end{equation}
If so, then ${\rm supp} [u]\cap {\rm supp} [v]=\emptyset$, $u,v\in\D(\EE)$, implies ${\rm supp}[u\circ \Psi]\cap{\rm supp} [v\circ \Psi]=\emptyset,$ so that  \eqref{EE'} and   the local property of $(\widetilde \EE, \D(\widetilde\EE))$  yield
$$\EE(u,v)=  \widetilde \EE (u\circ\Psi, v\circ\Psi)=0.$$
Let $\phi\in T_0\setminus \Psi^{-1}(\supp[u])$. 
There exists an open neighborhood $U\subseteq \scr P_p$ of $\Psi(\phi)$ such that 
$$0=\int_U |u|\,\d\LL= \int_{\Psi^{-1}(U)} |u\circ \Psi|\,\d\LL_0.$$
So, $\phi\in T_0 \setminus  {\rm supp} [u\circ\Psi]$ follows from continuity of $\Psi$ and $\phi\in \Psi^{-1}(U)$.
\end{proof}

We now define the notion of directional derivatives on $\scr P_p$.

\begin{rem}\label{rem:sect}
	Let $Q:L^2(T_0\to T_{\mu_0,2},\LL_0)\to L^2(T_0\to T_{\mu_0,2},\LL_0)$ be decomposable as in \eqref{eq:dec} with
	\begin{equation*}
		Q_\phi:= {\langle \eta,\,\cdot\, \rangle}_{T_{\mu_0,2}}\eta,\qquad \phi\in T_0,
	\end{equation*}
	for some fixed $\eta\in T_{\mu_0,2}$. With $P$ and adjoint $P^*$ as defined in the proof of Theorem \ref{T2}, the definition $B:=P^*QP$ yields a symmetric bounded decomposable operator
	$B={\{B_\mu\}}_{\mu\in\scr P_p}$ on ${\textstyle\int}^\oplus_{\scr P_p} T_{\mu,2}\d\LL(\mu)$ and the range of $B_\mu$ has dimension smaller equal to $1$. Hence, there exists 
	an element $V\in {\textstyle\int}^\oplus_{\scr P_p} T_{\mu,2}\d\LL(\mu)$, which we denote by $V:={(\eta_\mu)}_{\mu\in \scr P_p}$, such that
	\begin{equation}\label{eq:Beta}
		BW(\mu)=\langle W(\mu),\eta_\mu\rangle_{T_{\mu,2}}\eta_\mu,\qquad \mu\in\scr P_p,\, W\in  {\textstyle\int}^\oplus_{\scr P_p} T_{\mu,2}\d\LL(\mu).
	\end{equation}
\end{rem}
	
\begin{lem}\label{lem:partial}Let $\eta$, ${(\eta_\mu)}_{\mu\in \scr P_p}$ be as in Remark \ref{rem:sect} and $u\in C_b^1(\scr P_p)$. For $\LL_0$-a.e.~$\phi\in T_0$ it holds
	\begin{equation*}
		\partial_\eta (u\circ\Psi)(\phi)= (D_{\eta_{\Psi(\phi)}}u)(\Psi(\phi)).
	\end{equation*}
	Moreover, for $\LL_0$-a.e.~$\phi\in T_0$  it holds
	\begin{equation*}
		\langle \eta,\xi\circ\phi\rangle_{T_{\mu_0,2}}=\langle \eta_{\Psi(\phi)}\circ\phi,\xi\circ\phi\rangle_{T_{\mu_0,2}},\qquad 
		\xi\in T_{\Psi(\phi),2}.
	\end{equation*}
\end{lem}
\begin{proof}
	Let $Q$, $B$, $P$ and $P^*$ be as in Remark \ref{rem:sect} and $V:=Du$. %The constant vector field $T_0\ni \phi \mapsto \eta\in T_{\mu_0,2}$ is denoted again by $\eta$. 
	
	Since $P:{\textstyle\int}^\oplus_{\scr P_p} T_{\mu,2}\d\LL(\mu)\to L^2(T_0\to T_{\mu_0,2},\LL_0)$ is isometric, 
	$P^*P$ coincides with the identity on ${\textstyle\int}^\oplus_{\scr P_p} T_{\mu,2}\d\LL(\mu)$ and hence
	$PP^*$ coincides with the orthogonal projection in
	$L^2(T_0\to T_{\mu_0,2},\LL_0)$ onto the range of $P$. 
	We identify $\eta$ with the constant vector field $\eta \in L^2(T_0\to T_{\mu_0,2},\LL_0)$.
	For every function $g\in L^2(T_0,\LL_0)$ which is measurable w.r.t.~$\sigma(\Psi)$, the $\sigma$-algebra generated by $\Psi$, the definitions of $P$ and $P^*$ yield
	$PP^*(g\eta)=gPP^*\eta$.
	Moreover, $PV=\nabla (u\circ\Psi)$ and hence the scalar function \begin{equation*}T_0\ni \phi\mapsto \langle PV(\phi),\eta\rangle_{T_{\mu_0,2}}\end{equation*} is measurable w.r.t.~$\sigma(\Psi)$.
	This, together with $PBV=PP^*QPV$ yields
	\begin{align*}
		\langle V(\Psi(\phi)),\eta_{\Psi(\phi)}\rangle_{T_{\Psi(\phi),2}}\eta_{\Psi(\phi)}\circ\phi&=
	PP^*\big(	\langle PV,\eta\rangle_{T_{\mu_0,2}}\eta\big)(\phi)\\&=\langle PV(\phi),\eta\rangle_{T_{\mu_0,2}}PP^*\eta(\phi)
	\end{align*}
	for $\LL_0$-a.e.~$\phi\in T_0$.
	Now, since 
	\begin{equation*}
		\int_{T_0}\langle PV(\phi),PP^*\eta(\phi)\rangle_{T_{\mu_0,2}}g(\phi)\d\Lambda_0(\phi)=\int_{T_0}\langle PV(\phi),\eta\rangle_{T_{\mu_0,2}}g(\phi)\Lambda_0(\d\phi)
	\end{equation*}
	for every $\sigma(\Psi)$-measurable function $g\in L^2(T_0,\LL_0)$, we obtain
	\begin{multline*}
		\langle PV(\phi),PP^*\eta(\phi)\rangle_{T_{\mu_0,2}}PP^*\eta(\phi)=
		\langle PV(\phi),\eta\rangle_{T_{\mu_0,2}}PP^*\eta(\phi)\\
		=\langle V(\Psi(\phi)),\eta_{\Psi(\phi)}\rangle_{T_{\Psi(\phi),2}}\eta_{\Psi(\phi)}\circ\phi
		=\langle PV(\phi),\eta_{\Psi(\phi)}\circ\phi\rangle_{T_{\mu_0,2}}\eta_{\Psi(\phi)}\circ\phi,
	\end{multline*}
	$\LL_0$-a.e.~$\phi\in T_0$. Since the above equality holds for every $u\in C_b^1(\scr P_p)$, $V=Du$, we must have either
	$PP^*\eta(\phi)=\eta_{\Psi(\phi)}\circ\phi$, or $PP^*\eta(\phi)=-\eta_{\Psi(\phi)}\circ\phi$.
	W.l.o.g.~we may assume the former, since otherwise we can redefine ${(\eta_\mu)}_{\mu\in\scr P_p}$ in Remark \ref{rem:sect} accordingly without compromising \eqref{eq:Beta}.
	For $u\in C_b^1(\scr P_1)$ and $\LL_0$-a.e.~$\phi\in T_0$ we have shown
	\begin{multline*}
		\partial_\eta (u\circ\Psi)(\phi)=\langle \nabla(u\circ\Psi)(\phi),\eta\rangle_{T_{\mu_0,2}}=\langle PDu(\phi),\eta\rangle_{T_{\mu_0,2}}
		=\langle PDu(\phi),PP^*\eta(\phi)\rangle_{T_{\mu_0,2}}\\
		=\langle PDu(\phi),\eta_{\Psi(\phi)}\circ\phi\rangle_{T_{\mu_0,2}}=\langle Du(\Psi(\phi)),\eta_{\Psi(\phi)}\rangle_{T_{\Psi(\phi),2}}= D_{\eta_{\Psi(\phi)}}u(\Psi(\phi)).
	\end{multline*}
	The second statement is a consequence of $PP^*\eta(\phi)=\eta_{\Psi(\phi)}\circ\phi$.
\end{proof}

Let $\eta\in T_0$ and $\tau_{\eta,s}:T_0\ni\phi \mapsto \phi+s\eta\in T_0$.
The measure $\LL_0$ is called quasi-shift invariant w.r.t.~$\eta$  if  the push-forward $\LL_0\circ\tau_{\eta,s}^{-1}$ is absolutely continuous w.r.t.~$\LL_0$ for all $s\in\R$.
The next example defines the standard-gradient form on $\scr P_p$ and its component forms.

\begin{exa}
	 (i) Let $\eta\in T_{\mu_0,2}$. We assume that $\LL_0$ is quasi-shift invariant w.r.t.~$\eta$ and the Radon-Nikodym derivative $\d(\LL_0\circ\tau_{\eta,s}^{-1})/\d\LL_0$  is continuous
	 for $s\in\R$.
	  In view of  the closability results in \cite{AR} together with Theorem 
	 \ref{T2}, Remark \ref{rem:sect} and Lemma \ref{lem:partial}, there exists ${(\eta_{\mu})}_{\mu\in\scr P_p}\in{\textstyle\int}^\oplus_{\scr P_p} T_{\mu,2}\d\LL(\mu)$
	 such that
	 \begin{equation*}
	 	\EE_{\eta}(u,v):=\int_{\scr P_p} D_{\eta_\mu} u(\mu)D_{\eta_\mu} u(\mu)\d\LL(\mu)\quad\text{for }u,v\in C_b^1(\scr P_p)
	 \end{equation*}
	 defines a closable form in $L^2(\scr P_p,\LL)$ and satisfies $\EE_{\eta}(u,v)=\int_{T_0}\partial_{\eta}(u\circ\Psi)\partial_{\eta}(v\circ\Psi)\d\LL_0$.
	 
	 (ii) If there exists an orthonormal basis $\{\phi_n\}_{n\in\N}$ of $T_{\mu_0,2}$ such that $\LL_0$ is quasi-shift invariant w.r.t.~each $\phi_n$ and the respective density of $\LL_0\circ{\tau_{\phi_n,s}^{-1}}$ is continuous, then  
	 the application of Theorem \ref{T2} regarding the standard gradient form
	 \begin{equation*}
	 	\tt\EE(f,g)=\sum_{n\in\N}\int_{T_0} \partial_{\phi_n} f(\phi),\partial_{\phi_n} g(\phi)\d\LL_0(\phi),\qquad f,g\in C_b^1(T_0),
	 \end{equation*}
	 together with Remark \ref{rem:sect} and Lemma \ref{lem:partial} yield the following:
	 There exist $\phi_{\mu,n}\in T_{\mu,2}$ for $\mu\in \scr P_2$, $n\in\N$, such that
	  \begin{equation*}
	  	\EE(u,v):=\int_{\scr P_p}\Gamma (u,v)\d\LL(\mu),\qquad u,v\in C_b^1(\scr P_p),
	  \end{equation*}
	  with square-field operator 
	  \begin{equation}\label{eq:Gsum}
	  	\Gamma(u,v)(\mu)=\sum_{n\in\N} D_{\phi_{\mu,n}} u(\mu) D_{\phi_{\mu,n}} v(\mu)
	  	\end{equation}
	  is well-defined and closable in $L^2(\scr P_p,\LL)$. Its closure $(\EE,\D(\EE))$ is local, quasi-regular Dirichlet form and satisfies
	 \begin{equation*}
	 \EE(u,v)=\sum_{n\in\N}\EE_{\phi_n}(u,v),\qquad u,v\in\D(\EE).
	 \end{equation*}
	 
	 (iii) Since 
	 \begin{equation*}
	 	\sum_{n\in\N}\partial_{\phi_n} f(\phi)\partial_{\phi_n} g(\phi)=\langle \nabla f(\phi),\nabla g(\phi)\rangle_{T_{\mu_0,2}}
	 \end{equation*}
	 for $f,g\in C_b^1(T_0)$, $\LL_0$-a.e.~$\phi\in T_0$, the square-field operator $\Gamma$ defined in \eqref{eq:Gsum} satisfies
	 \begin{equation*}
	 \Gamma(u,v)=\langle D u(\mu), D v(\mu)\rangle_{T_{\mu,2}},\qquad\LL\text{-a.e.~}\mu\in\scr P_p,\,u,v\in C_b^1(\scr P_p).
	\end{equation*}
\end{exa}

Theorem \ref{T2} ensures the existence of a diffusion process on $\scr P_p$.
In \cite{Fu} a correspondence between regular Dirichlet forms and strong Markov processes is built,
see  \cite{FOT11} for a complete theory and more references. This is extended in \cite{AM}, \cite{MR92} to the quasi-regular setting.  According to \cite{CMR} a quasi-regular Dirichlet form becomes regular under one-point compactification.

According to  \cite[Definitions IV.1.8,  IV.1.13,  V.1.10]{MR92}   a standard Markov process $$\mathbf M=\big(\Omega,\scr F,(X_t)_{t\geq 0},(\P_\mu)_{\mu\in\scr P_p} \big)$$  with natural filtration $\{\scr F_t\}_{t\geq 0}$ is called a non-terminating diffusion process on $\scr P_p$ if
$$\P_\mu\big(X_\cdot\in C([0,\infty),\scr P_p)\big)=1\quad\text{for } \mu\in \scr P_p.$$
It is called $\LL$-tight if there exists a sequence $\{K_n\}_{n}$  of compact sets in $\scr P_p$ such that the stopping times $$\tau_n:=\inf\{t\ge 0: X_t\notin K_n\},\ \ n\in\N$$ satisfy
$$\P_\mu\big(\lim_{n\to\infty} \tau_n=\infty\big)=1\quad\text{for }\LL\text{-a.e.~}\mu\in \scr P_p.$$
The diffusion process is called properly associated with $(\EE,\D(\EE))$, if for any bounded measurable function $u:\scr P_p\to\R$ and $t> 0$,
$$\scr P_p\ni \mu\mapsto\int_{\Omega}u(X_t)\,\d\P_\mu$$
is a quasi-continuous $\LL$-version of $T_t u$, where $(T_t)_{t\ge 0}$ is the Markov semigroup on $L^2(\scr P_p,\LL)$ corresponding to $(\EE,\D(\EE))$.

\begin{cor}\label{C1} In the situation of Theorem $\ref{T2}$, we have the following assertions.
\beg{enumerate}
\item[$(1)$] 	There exists a non-terminating diffusion process $\mathbf M=(\Omega,\scr F,(X_t)_{t\geq 0},(\P_\mu)_{\mu\in\scr P_p} )$ on $\scr P_p$ which is  properly associated with $(\EE,\D(\EE))$.  In particular,  $\LL$ is an invariant probability measure of  $\mathbf M$.
\item[$(2)$] $\mathbf M$  solves the martingale problem for the generator $(L,\D(L))$ of $(\EE,\D(\EE))$, i.e.
for $u\in\D(L)$, the additive functional
\begin{equation*}
\tilde u (X_t)-\tilde u(X_0)-\int_0^tLu(X_s)\d s,\quad t\geq 0,
\end{equation*}
is an $\{\scr F_t\}_{t}$-martingale under $\P_\mu$ for q.e.~$\mu\in\scr P_p$, where $\tilde u$ denotes a quasi-continuous $\LL$-version of $u$.
\end{enumerate}
\end{cor}

\beg{proof}
(1) By \cite[Theorem IV.3.5 \&  Theorem  V.1.11]{MR92}, the locality and quasi regularity ensured by Theorem \ref{T2} imply the existence of
a $\LL$-tight special standard process $$\mathbf M=(\Omega,\scr F,(X_t)_{t\geq 0},(\P_\mu)_{\mu\in\scr P_p\cup \{\Delta\}} )$$ with state space $(\scr P_p,\W_p)$, life time $\zeta$ and filtration $\{\scr F_t\}_{t\geq 0}$
(as defined in \cite[Chap.~IV, Definition IV.1.5, IV.1.8, IV.1.13, V.1.10]{MR92})
which meets
\begin{equation*}
\P_\mu(\{\omega\in\Omega\,:\,[0,\zeta(\omega))\ni t\mapsto X_t(\omega)\text{ is continuous}\})=1\quad\text{for }\mu\in \scr P_p.
\end{equation*}
and is properly associated with $(\EE,\D(\EE))$ in the sense of \cite[Definition IV.2.5]{MR92}.
%The semigroup on $L^2(\scr P_2,G_{\mu_0,Q})$ corresponding to \sloppy $(\EE^{\mu_0,G},\D(\EE^{\mu_0,G}))$ is denoted by $\{T_t\}_{t\geq 0}$. With the argument from \cite[Theorem~3.2]{RW22} it follows that $T_t$ is a compact (symmetric) operator on $L^2(\scr P_2,G_{\mu_0,Q})$.

Since $\eins_{\scr P_p}\in\D(\EE)$ and $\EE(\eins_{\scr P_p},\eins_{\scr P_p})=0$, it holds $T_t\eins_{\scr P_p},=\eins_{\scr P_p}$ for $t\geq 0$. This means there exists a set $N\subseteq \scr P_p$ of zero capacity (referring to the $1$-capacity associated with $\EE$) such that  $\P_\mu(\{\zeta=\infty\})=1$ for $\mu\in\scr P_p\setminus N$. Without loss of generality,  $\scr P_p\setminus N$ may be assumed to be $\mathbf M$-invariant, by virtue of \cite[Corollariy IV.6.5]{MR92}. Considering the restriction $\mathbf M|_{\scr P_p\setminus N}$ (see \cite[Remark IV.6.2(i)]{MR92}) and then applying
the procedure described in \cite[Chapt.IV, Sect.3, pp.~117f.]{MR92}, re-defining $\mathbf M$ in such way that each element from $N$ is a trap,  we may assume $\P_\mu(\{\zeta=\infty\})=1$ for all $\mu\in\scr P_p$.
Furthermore, after the procedure of weeding (restricting the sample space to a subset of $\Omega$), as explained in \cite[Chap.~III, Paragraph 2, pp.~86f.]{D65}, we may  assume that $\mathbf M$ is non-terminating and continuous, i.e.~$\zeta(\omega)=\infty$ and $[0,\infty)\mapsto X_t(\omega)$ a continuous map for every $\omega\in\Omega$.

(2) 	Let $u\in\D(L)$ and
\begin{equation*}
A_t:\Omega\ni\omega\mapsto\int_0^tLu(X_s(\omega))\d s,\quad t\geq 0.
\end{equation*}
Then, $\{A_t\}_{t\geq 0}$ is a continuous additive functional of $\mathbf M$ with zero energy. Moreover,
\begin{equation*}
\E_\mu(A_t)=\int_0^t(T_sLu)(\mu)\d s=(T_tu-u)(\mu)\quad \text{for }\LL\text{-a.e.~}\mu\in\scr P_p.
\end{equation*}
Now, the claim follows from \cite[Theorem VI.2.5]{MR92}, resp.~\cite[Theorem 5.2.2]{FOT11}, in combination with \cite[Theorem 5.2.4]{FOT11} and regularization of the Dirichlet form $(\EE,\D(\EE))$ as explained in \cite[Chap.~VI]{MR92}.
\end{proof}

\section{  Ornstein--Uhlenbeck type processes}\label{sec:applic}

\noindent In this section we study O-U type Dirichlet forms with the methods of Section \ref{sec:gradForm}.
Here, $\LL_0$ is chosen as  a non-degenerate Gaussian measure on the tangent space $T_0$. We  consider the case 
in which $X=H$ is a separable Hilbert space and $p=2$.
The framework in \cite{RW22}, where $H=\R^d$, 
can be treated analogously.

\subsection{O-U type process on \texorpdfstring{$\scr P_2$}{P2} over a Hilbert space }\label{sec:exaFTS}

\noindent Let $X=H$ be an infinite-dimensional, separable Hilbert space, for which we consider the quadratic Wasserstein space $\scr P_2$.
We fix $\mu_0\in\scr P_2$ and denote the tangent space by  $T_0:=T_{\mu_0,2}=L^2(H\to H,\mu_0)$ in accordance with the notation in Section \ref{sec:gradForm}.
Furthermore, it is assumed that $\mu_0$ is absolutely continuous w.r.t.~a non-degenerate Gaussian measure on $H$.
As a consequence of \cite[Thm.~6.2.10]{AGS05}, the map
\begin{equation*}
	\Psi:T_0\ni \phi\mapsto \mu_0\circ\phi^{-1}\in\scr P_2
\end{equation*}
is surjective. In fact, it is enough to assume that $\mu_0$ is regular in the sense of Optimal Transport (see \cite[Section 6.2]{AGS05}).

The following application of the general results in Section \ref{sec:gradForm},  in particular the conclusion in Corollary \ref{cor:heatKenel}, are also valid in the finite-dimensional case $X=\R^d$ by an analogous argumentation. In the finite-dimensional case the assumption on $\mu_0$ simply translates into $\mu_0$ being absolutely continuous w.r.t.~the Lebesgue measure.

Let $(A,\D(A))$ be a strictly positive definite, self-adjoint linear operator on $T_0$ with pure point spectrum
and ${\{\alpha_n\}}_{n\in\N}\subset(0,\infty)$ denote its eigenvalues 
listed with multiplicity in the increasing order.
The corresponding normalized eigenvectors  $\{\phi_n\}_{n\in\N}$ form an orthonormal basis of $T_0$. %, which is called the eigenbasis of $(A,\D(A))$.
We assume that
$$\sum_{n\in\N} \alpha_n^{-1}<\infty,$$
which ensures the existence of a non-degenerate, centered Gaussian measure $G$ on $T_0$ whose covariance operator is given by the inverse of $A$. In the following, we identify $T_0$ with $\ell^2$ using the coordinate representation w.r.t.~$\{\phi_n\}_{n\in\N}$, i.e.
\begin{equation*}
T_0\ni\phi\overset{\simeq}{\longmapsto}{\big(\langle\phi_n,\phi\rangle_{T_0}\big)}_{n\in\N}\in \ell^2.
\end{equation*}
Then, the Gaussian measure $G$  coincides with the product measure
\beq\label{MN} G(\d \phi):= \prod_{n\in \N} m_n(\d \<\phi_n,\phi\>_{T_0})\quad \text{with}\quad m_n(\d r):=  \Big(\ff{\alpha_n}{2\pi}\Big)^{\ff 1 2} \exp\Big[-\ff{\alpha_nr^2}{2}\Big]\d r.\end{equation}
In \cite{RW22} the  corresponding non-degenerate Gaussian measure on $\scr P_2$   is defined as
\begin{equation*}
N_G:=  G\circ \Psi^{-1}.
\end{equation*}
$N_G$ has full topological support, since $\Psi$ is surjective and continuous. The next remark demonstrates that the same measure $N_G$
 can be constructed in the analogous way starting with any other suitable choice  $\mu_*\in \scr P_2$ (i.e.~$\mu^*$ is regular in the sense of Optimal Transport),
instead of $\mu_0$.
\begin{rem}\label{rem:inv}
Choosing $\phi_*\in T_{\mu^*}=L^2(H\to H,\mu_*)$ such that $\mu_0=\mu_*\circ\phi_*^{-1}$ we obtain the isometry
\begin{equation*}
I:T_{0}\ni\phi\mapsto \phi\circ\phi_*\in T_{\mu_*}.
\end{equation*}
Let
\begin{equation*}
\Psi_*:T_{\mu_*}\ni\phi\mapsto\mu_*\circ\phi^{-1}\in \scr P_2\quad\text{and}\quad G_*:= G\circ I^{-1}.
\end{equation*}
Then, $G_*$ is a Gaussian measure on $T_{\mu_*}$ and
\begin{equation*}
N_G=G\circ\Psi^{-1}=G\circ(\Psi_*\circ I)^{-1}=G_*\circ\Psi_*^{-1}
\end{equation*}
yields a representation of $N_G$ in terms of $\mu_*$ and $G_*$. 
\end{rem}
By  \cite[Theorem~3.10]{AR}, the bilinear form
\begin{equation*}
\tt\EE(f,g):= \LL_0(\<\nn f,\nn g\>_{T_0}),\quad f,g\in C_b^1(T_0),
\end{equation*}
is closable in $L^2(T_0,\LL_0)$ and its closure $(\tt\EE,\D(\tt\EE))$ is a local Dirichlet form. 
Moreover, by \cite[Proposition~3.2]{RZ},  the class of smooth cylindrical functions
\begin{equation*}
\scr F C_b^\infty(T_0):=\big\{g(\<\cdot,\psi_1\>_{T_0},\cdots, \<\cdot, \psi_n\>_{T_0})\,:\,\ n\in \mathbb N,\ g\in C_b^\infty(\R^n),\,\psi_1,\dots,\psi_n\in T_0\big\} 
\end{equation*}
is  dense in $\D(\widetilde \EE)$ w.r.t.~$\widetilde \EE_1^{1/2}$-norm, so $(\tt\EE,\D(\tt\EE))$ is also the closure of $(\tt\EE, \scr F C_b^\infty(T_0)).$

By Theorem \ref{T2},  the bilinear form
\begin{equation*}
\EE(u,v):=\int_{\scr P_2} \big\<Du(\mu), D v(\mu)\big\>_{T_{\mu,2}}\,\d N_G(\mu),\quad u,v\in {C_b^1}(\scr P_2),
\end{equation*} is closable in
$L^2(\scr P_2,N_G)$, and the closure $(\EE,\D(\EE))$ is a
local, quasi-regular  Dirichlet form. 
%Moreover, as shown in  \cite[Theorem 3.2]{RW22}, $(\EE,\D(\EE))$ satisfies the log-Sobolev inequality has a semigroup of compact operators.
Together with an application of Theorem \ref{TN}(2) we have the following  result.

\begin{cor}\label{cor:heatKenel} $(\EE,\D(\EE))$ is a local, quasi-regular Dirichlet form in $L^2(\scr P_2,N_G)$.  Its generator $(L,\D(L))$ has a purely discrete spectrum.
	Denoting the eigenvalues of $-L$ by ${\{\lambda_n\}}_{n\in\N}\subset[0,\infty)$ (listed with multiplicities  in the increasing order) the associated Markov semigroup  $T_t:=\e^{t L}$ has a density $p_t$ w.r.t.~$N_G$ which satisfies the estimate
\beg{equation*} \int_{\scr P_2\times \scr P_2} p_t(\mu,\nu)^2 \,\d N_G(\mu) \,\d N_G(\nu)=\sum_{n\in\N} \e^{-2\ll_n t}
\le  \prod_{n\in \mathbb N}  \Big(1+ \ff{2\e^{-2\alpha_n t}}{(2\alpha_n t)\land 1}\Big)<\infty,\quad t>0.\end{equation*}
\end{cor}

\beg{proof} In view of Theorem  \ref{T2} it suffices to verify the estimate for $p_t$, $t>0$.
For $n\in\mathbb N$  we consider the O-U process on $\R$ generated by
$$L_n^{(1)}h(x):= h''(x)- \alpha_n x h'(x),\qquad x\in\R,\,h\in C_b^2(\R).$$
It is well known that $-L_n^{(1)}$ has eigenvalues $k \alpha_n$ for $k=0,1,\dots$, and the corresponding eigenfunctions are the Hermit polynomials
\begin{equation*}
	H_k(x):= \e^{ \alpha_nx^2/2} \ff{\d^k}{\d x^k} \e^{-\alpha_nx^2/2}.
\end{equation*}
Let $p_t^n(x,y)$ be the heat kernel w.r.t.~$m_n$  defined as in \eqref{MN}.
Then,  
\beg{align}\label{*R} \nonumber\int_{\R\times\R} p_t^n(x,y)^2\d m_n(x)\d m_n(y)&=\sum_{k=0}^\infty \e^{-2 k \alpha_n t} \\
 &\le 1+ \e^{-2 \alpha_nt}+\int_1^\infty   \e^{-2\alpha_nts} \d s  \le 1+ \ff{2\e^{-2\alpha_n t}}{(2\alpha_n t)\land 1}
  \end{align}
for  $n\in\mathbb N$ and $t>0$.
The assumption $\sum_{n\in\N} \alpha_n^{-1}<\infty$ implies
\begin{equation*}
\sum_{n\in\N} \log\Big(1+ \ff{2\e^{-2\alpha_n t}}{(2\alpha_n t)\land 1}\Big) \le \sum_{n\in\N}   \ff{2\e^{-2\alpha_n t}}{(2\alpha_n t)\land 1}<\infty.
\end{equation*}
With \eqref{*R} we conclude that
\begin{equation*}
p_t^\infty({\bf x}, {\bf y}):= \prod_{n\in\N} p_t^n(x_n,y_n),\quad {\bf x}=(x_n), {\bf y}=(y_n)\in \R^\N,
\end{equation*}
is well defined as an element of $L^2(m^{\infty}\times m^{\infty})$, where $m^{\infty}:= \prod_{n\in\N} m_n$ and
\begin{equation*}
\int   p_t^\infty({\bf x}, {\bf y})^2\d m^{\infty}({\bf x})  \d m^{\infty}({\bf y})=  \prod_{n\in\N}  \int_{\R\times\R} p_t^n(x,y)^2 \d m_n( x)\d m_n(y)\le\tilde \xi_t,
\end{equation*}
holds for
\begin{equation*}
\tilde\xi_t:=  \prod_{n\in \mathbb N}  \Big(1+ \ff{2\e^{-2\alpha_n t}}{(2\alpha_n t)\land 1}\Big)<\infty,\ \ t>0. 
\end{equation*}
Let ${\{\widetilde T_t\}}_{t>0}$ and $(\tilde L,\D(\tilde L))$ denote the semigroup, respectively the generator associated with  $(\widetilde\EE, \D(\widetilde\EE))$.  Then, 
$\tilde L$ has a purely discrete spectrum and
 $\widetilde T_t$  has a density with respect to $G$ which reads
\begin{equation*}
\tt p_t(\phi,\phi') = p_t^\infty({\bf x}(\phi), {\bf x}(\phi')),\ \  {\bf x}(\,\cdot\,):= {(\<\,\cdot\,,\phi_n\>_{T_{0}})}_{n\in\mathbb N}.
\end{equation*}
By the spectral representation (see for instance \cite{Davies})  the eigenvalues ${\{\sigma_n\}}_{n\in\N}$ of $-\widetilde L$, listed with multiplicities in the increasing order, satisfy
\begin{equation*}\sum_{n\in\N} \e^{-2 \sigma_nt}  = \int_{T_{0}\times T_{0}} \tt p_t(\phi,\phi')^2  \,\d G(\phi)\,\d G(\phi') \le\tilde \xi_t.
\end{equation*}
Now, the claim follows from Theorem \ref{TN}(2).
\end{proof}

\subsection{Partial integration for \texorpdfstring{$N_G$}{NG}}\label{sec:ibpf}

We continue in the setting of Section \ref{sec:exaFTS} with $G$, ${\{\phi_n\}}_{n\in\N}$ and $N_G$ as above.
By Lemma \ref{lem:partial} there exist $\phi_{\mu,n}\in T_{\mu,2}$ for $\mu\in\scr P_2$, $n\in\N$, such that for
 $u\in C_b^1(\scr P_2)$, $G$-a.e.~$\phi\in T_0$ we have
\begin{equation*}
	\partial_{\phi_n} (u\circ\Psi)(\phi)= (D_{\phi_{\Psi(\phi)},n}u)(\Psi(\phi))
\end{equation*}
and
\begin{equation*}
	\langle \phi_n,\xi\circ\phi\rangle_{T_0}=\langle \phi_{\Psi(\phi),n}\circ\phi,\xi\circ\phi\rangle_{T_0},\qquad \xi\in T_{\Psi(\phi),2}.
\end{equation*}
We now state a partial integration formula for $N_G$.

\begin{lem}\label{lem:PI}
	(i) For $n\in\N$, $u,v\in C_b^1(\scr P_2)$ it holds
	\begin{equation*}
		\int_{\scr P_2}D_{\phi_{\mu,n}}u(\mu)v(\mu)\d N_G(\mu)=-\int_{\scr P_2}u(\mu)\big(D_{\phi_{\mu,n}}v(\mu)- \alpha_nv(\mu)\langle\phi_{\mu,n},id\rangle_{T_{\mu,2}}\big)\d N_G(\mu).
	\end{equation*}
	
	\noindent (ii) Let $u\in C_b^1(\scr P_2)$ and  $V(\mu):=\sum_{n=1}^Mv_n(\mu)\phi_{\mu,n}$, $\mu\in\scr P_2$, for given  $M\in\N$, $v_n\in C_b^1(\scr P_2)$. 
	It holds
	\begin{equation*}
		\int_{\scr P_2} \langle Du(\mu),V(\mu)\rangle_{T_{\mu,2}}\d N_G(\mu)=-\int_{\scr P_2} u(\mu)\Big(\sum_{n=1}^MD_{\phi_{\mu,n}} v_n(\mu)
		-\alpha_n\langle \phi_{\mu,n},id\rangle_{T_{\mu,2}} v_n(\mu)\Big)\d N_G(\mu).
	\end{equation*}
\end{lem}
\begin{proof}
	(i) On the one hand, we have
	\begin{equation}\label{eq:oneH}
		\int_{T_0}\partial_{\phi_{n}}(u\circ\Psi)(\phi)(v\circ\Psi)(\phi)\d G(\phi)=
		-\int_{T_0}(u\circ\Psi)(\phi)\big(\partial_{\phi_{n}}(v\circ\Psi)(\phi)- \alpha_n(v\circ\Psi)(\phi)\langle\phi_{n},\phi\rangle_{T_0}\big)\d G(\phi)
	\end{equation}
	by construction of $G$. On the other hand, by Lemma \ref{lem:partial} it holds
	\begin{align*}
		\int_{T_0}\partial_{\phi_{n}}(u\circ\Psi)(\phi)(v\circ\Psi)(\phi)\d G(\phi)&=\int_{T_0}D_{\phi_{\Psi(\phi),n}}u(\Psi(\phi))v(\Psi(\phi))\d G(\phi)\\
		&=\int_{\scr P_2}D_{\phi_{\mu,n}}u(\mu)v(\mu)\d N_G(\mu)
	\end{align*}
	and 
	\begin{multline}\label{eq:twoH}
		\int_{T_0}(u\circ\Psi)(\phi)\big(\partial_{\phi_{n}}(v\circ\Psi)(\phi)- \alpha_n(v\circ\Psi)(\phi)\langle\phi_{n},\phi\rangle_{T_0}\big)\d G(\phi)=
		\\\int_{T_0}(u\circ\Psi)(\phi)\big(D_{\phi_{\Psi(\phi),n}}v(\Psi(\phi))- \alpha_n(v\circ\Psi)(\phi)\langle\phi_{\Psi(\phi),n},id\rangle_{T_{\Psi(\phi),2}}\big)\d G(\phi)\\=
		\int_{\scr P_2}u(\mu)\big(D_{\phi_{\mu,n}}v(\mu)- \alpha_nv(\mu)\langle\phi_{\mu,n},id\rangle_{T_{\mu,2}}\big)\d N_G(\mu).
	\end{multline}
	
	(ii) Again by Lemma \ref{lem:partial} we have
	\begin{align*}
			\int_{\scr P_2} \langle Du(\mu),V(\mu)\rangle_{T_{\mu,2}}\d N_G(\mu)&=
			\sum_{n=1}^M\int_{T_0}v_n(\Psi(\phi)) \langle Du(\Psi(\phi))\circ\phi,\phi_{\Psi(\phi),n}\circ\phi\rangle_{T_0}\d G(\phi)\\
			&=\sum_{n=1}^M\int_{T_0} v_n(\Psi(\phi))\langle Du(\Psi(\phi))\circ\phi,\phi_{n}\rangle_{T_0}\d G(\phi).
	\end{align*}
	Applying \eqref{eq:oneH} and \eqref{eq:twoH} yields
	\begin{align*}
		\int_{\scr P_2} \langle Du(\mu),V(\mu)\rangle_{T_{\mu,2}}\d N_G(\mu)=-\sum_{n=1}^M
		\int_{\scr P_2}u(\mu)\big(D_{\phi_{\mu,n}}v_n(\mu)- \alpha_nv_n(\mu)\langle\phi_{\mu,n},id\rangle_{T_{\mu,2}}\big)\d N_G(\mu)
	\end{align*}
	as claimed.
\end{proof}

\begin{rem}
Let $u\in C_b^1(\scr P_2)$. For $V\in {\textstyle\int}^\oplus_{\scr P_2} T_{\mu,2}\d N_G(\mu)$ we have
\begin{multline*}
	\sum_{n\in\N}\int_{\scr P_2}\langle  D_{\phi_{\mu,n}} u(\mu)\phi_{\mu,n}, V(\mu)\rangle_{T_{\mu,2}}\d N_G(\mu)
	=\sum_{n\in\N}\int_{T_0}\langle \partial_{\phi_{n}} u(\Psi(\phi))\phi_{n}, V(\Psi(\phi))\circ\phi\rangle_{T_{0}}\d G(\phi)\\
	=\int_{T_0}\langle \nabla(u\circ\Psi)(\phi), V(\Psi(\phi))\circ\phi\rangle_{T_{0}}\d G(\phi)
	=\int_{\scr P_2}\langle Du(\mu), V(\mu)\rangle_{T_{\mu,2}}\d N_G(\mu)
\end{multline*}
and hence
 \begin{equation*}
 	Du(\mu)=\sum_{n\in\N} D_{\phi_{\mu,n}}u(\mu)\phi_{\mu,n },\qquad N_G\textnormal{-a.e.~}\mu\in\scr P_2.
 \end{equation*}
% 
% (ii) Let $V=\sum_{n=1}^\infty\xi_n\phi_{\mu,n}\in L^2(T_0,G)$, $\mu\in\scr P_2$, for given $\xi_n\in C_b^1(\scr P_2)$ such that
% the series $$\sum_{n=1}^\infty \Big(D_{\phi_{\mu,n}}\xi_n(\mu)-\alpha_n\langle \phi_{\mu,n},id\rangle_{T_{\mu,2}} \xi_n(\mu)\Big),\,\mu\in \scr P_2,$$ converge in $L^2(\scr P_2,N_G)$.
%We can deduce from Lemma \ref{lem:PI}(ii) that 
% \begin{equation*}
% 	\int_{\scr P_2} \langle Du(\mu),V(\mu)\rangle_{T_{\mu,2}}\d N_G(\mu)=-\int_{\scr P_2} u(\mu)\Big(\sum_{n=1}^\infty D_{\phi_{\mu,n}} \xi_n(\mu)
% 	-\alpha_n\langle \phi_{\mu,n},id\rangle_{T_{\mu,2}} \xi_n(\mu)\Big)\d N_G(\mu).
% \end{equation*}
\end{rem}
Let $n\in\N$.
 We define $\partial_nu:\mu\mapsto D_{\phi_{\mu,n}}u(\mu)\in\R$ for $u\in C_b^1(\scr P_2)$
 and if $\partial_nu\in C_b^1(\scr P_2)$ we set
 \begin{equation*}
 L_nu:=\partial_n\partial_nu(\mu)
 -\alpha_n\langle \phi_{\mu,n},id\rangle_{T_{\mu,2}} \partial_nu(\mu).
  \end{equation*}
  As a consequence of Lemma \ref{lem:PI}(ii), for such $u$, the operator $L_n$ coincides with generator of the $n$-th component form $(\EE^n,\D(\EE^n))$, i.e.~the closure of
  \begin{equation*}
  	\EE^n(u,v):=\int_{\scr P_2}\partial_nu\partial_nv\d N_G, \qquad u,v\in C_b^1(\scr P_2),
  \end{equation*}
  in $L^2(\scr P_2,N_G)$.
By the partial integration formula  for $N_G$ stated in Lemma \ref{lem:PI}(ii)
and a straight-forward approximation, we can give an explicit expression for the generator $L$ of the Markov semigroup considered in Corollary \ref{cor:heatKenel}, 
on a certain set of functions.
\begin{cor} 
	If $u,\partial_nu\in C_b^1(\scr P_2)$ for  $n\in\N$ and the series
	$\displaystyle\sum_{n\in\N} L_nu$ converges in $L^2(\scr P_2,N_G)$, then $u\in D(L)$ and $\displaystyle Lu=\sum_{n\in\N} L_nu$.
\end{cor}

\section*{Acknowledgements}
\noindent Funding and supporting by Deutsche Forschungsgemeinschaft (DFG, German Research Foundation) – Project-ID 317210226 – SFB 1283 and Research Centre for Nonlinear Analysis at Hong Kong PolyU are gratefully acknowledged.

\end{document}